\documentclass[review]{elsarticle}
\usepackage{amsmath, amssymb, amsthm, mathabx, mathrsfs}

\usepackage{lineno,hyperref}
\modulolinenumbers[5]

\journal{Journal of \LaTeX\ Templates}

\newtheorem{theo}{Theorem}[section]

\newtheorem{lemma}{Lemma}[section]

\newtheorem{defi}{Definition}[section]

\numberwithin{equation}{section}

\begin{document}
	
\begin{frontmatter}
		
\title{Iteration Formulae for Brake Orbit and Index Inequalities for Real Pseudoholomorphic Curves}
		
\author[address]{Beijia Zhou\corref{mycorrespondingauthor}}
\cortext[mycorrespondingauthor]{Corresponding author}
\ead{beijiachow@gmail.com}

\address[address]{Chern Institute of Mathematics, Nankai University and LPMC, Tianjin 300071, China.}
		
\begin{abstract}
I give precise iteration formulae for brake orbit in dimension $3$ and use these formulae to get some index inequalities for moduli spaces of Real pseudoholomorphic Curves, which are important to establish Real embedded contact homology and Real cylindrical contact homology in dimension $3$.
\end{abstract}
		
\end{frontmatter}

\section{Introduction}

\subsection{Motivation}

In contact manifold of dimension 3, M.Hutchings used iteration formulae for periodic orbit to study moduli space of pseudoholomorphic curves and established embedded contact homology, which we abbreviate by ECH, and later with J.Nelson, they established cylindrical contact homology, which we abbreviate by CCH, for dynamically convex contact forms. See references \cite{Hutchings}, \cite{HT1}, \cite{HT2}, \cite{ECH}, \cite{HN}.

The origin of ECH comes from the Seiberg-Witten Floer Homology and the equivalence between Seiberg-Witten invariant and Gromov invariant\cite{Taubes1}, \cite{Taubes2}, \cite{Taubes3}, \cite{Taubes4}, \cite{Taubes5}, \cite{Taubes}. There exist some counterpart theories for Real pseudoholomorphic curves like Real Gromov-Witten invariant, defined by J.Welschinger\cite{Welschinger} and Real Seiberg-Witten invariant, defined by G.Tian and S.Wang\cite{Tian}. And Real pseudoholomorphic curves have been used to study brake orbits in contact manifold by U.Frauenfelder and J.Kang\cite{Urs}.

Following those ideas, I think there should exists Real cylindrical contact homology, Real embedded contact homology and Real Seiberg-Witten Floer Homology. At the same time Real embedded contact homology should be isomorphic to Real Seiberg-Witten Floer Homology by similar method of C.Taubes in \cite{Taubes}.

In order to construct Real ECH, we must know the relevant information of iteration formulae for brake orbits and moduli space of Real pseudoholomorphic curves. In this paper we will study these two aspects. It is the first step to construct Real ECH.

\subsection{Preliminaries}
Here we set up our conventions.

Let $(Y, \lambda)$ be a closed contact manifold of dimension $3$ with contact form $\lambda$, $\xi = \ker \lambda$ be the associate contact structure, and $R$ be the associate Reeb vector field. Let $N$ be an anticontact involution of $Y$, which means $N$ is an automorphism of $Y$, and $N^2 = Id, N^*\lambda = -\lambda$. A periodic Reeb orbit of period $\tau$ is a map $\alpha : {\bf R} \to Y$, which satisfies $\alpha(t + \tau) = \alpha(t), \alpha'(t) = R(\alpha(t))$. For a periodic Reeb orbit $\alpha$, we can choose a trivialization  of the symplectic space $(\xi, d\lambda)$ along $\alpha(t)$. So the linearized Reeb flow  $\Psi_{\alpha}(t) : (\xi_{\alpha(0)}, d\lambda) \to (\xi_{\alpha(t)}, d\lambda)$ is a symplectic path by the trivialization. A periodic Reeb orbit is nondegenerate if $\Psi_\alpha(\tau)$ does not have $1$ as an eigenvalue. In the following we assume all periodic orbits are nondegenerate. This condition holds for generic contact forms.

A brake orbit $\beta$ with period $\tau$ is a periodic orbit with Real symmetry i.e.
\begin{align}
\begin{cases}
\beta'(t) & = R(\beta(t)), \\
\beta(t + \tau) & = \beta(t), \\
\beta(-t) & = N\beta(t), \\
\end{cases}
\end{align}
Let $L := \{x \in Y|Nx = x \}$, $L$ is a Legendrian submanifold if $L \neq \emptyset$.
The above conditions are equivalent to
\begin{align}
\begin{cases}
\beta'  = R(\beta(t)), \\
\beta(0), \beta(\frac{\tau}{2}) \in L
\end{cases}
\end{align}

For any $M \in Sp(2n)$, its graph is defined by $Gr(M)  := \{ (x, Mx)| x \in {\bf R}^{2n}\}$ Let $L_0 := \{ 0 \} \times {\bf{R}}^n, L_1 := {\bf{R}}^n \times \{ 0 \}$,
$W := \{ (x, x) \in {{\bf{R}}^{4n}}| x \in {\bf R}^{2n} \}$. Let $\mu^{RS}$ be the Robbin-Salamon index for a pair of paths of Lagrangian subspace defined in \cite[\S 3]{Robbin}. We define $\mu_{CZ}(\alpha) = \mu^{RS}(W, Gr(\Psi([0, \tau])))$, and $\mu_1(\beta) = \mu^{RS}(L_0, \Psi([0, \frac{\tau}{2}])L_0)$, $\mu_2(\beta) = \mu^{RS}(L_1, \Psi([0, \frac{\tau}{2}])L_1)$, $\mu_1, \mu_2$ count only half period because of Real symmetry. In \cite[Proposition C]{LZZ}, they proved $\mu_{CZ} = \mu_1 + \mu_2, |\mu_1 - \mu_2| \leq n$.\footnote{Our convention is a little different to the one in \cite{LZZ}, because we use Robbin-Salamon index.}

In symplectization of $Y$, i.e. $({\bf R} \times Y, \omega = d(e^s\lambda))$, we can extend the involution $N$ to ${\bf R} \times Y$ by defining $N$ be identity map on $\bf R$. It is an antisymplectic involution $N^*\omega = -\omega$, and we still use $N$ to denote the extension. We can take an almost complex structure $J$ on $\xi$ compatible with $N, J \circ dN = -dN \circ J$, and extend it to ${\bf R} \times Y$ by defining $J\partial_s = R$, where $s$ is the coordinate of ${\bf R}$. We can choose a trivialization of $\xi$ such that $J, dN$ are represented by the following standard matrices in each fiber
\begin{equation*}
J_0 = \begin{pmatrix}
0 & -1 \\
1 & 0
\end{pmatrix}, N_0 = \begin{pmatrix}
-1 & 0 \\
0 & 1
\end{pmatrix}
\end{equation*}
In the following we always choose such trivialization.

A pseudoholomorphic curve is a map $u : (\Sigma, i) \to ({\bf R} \times Y, J), du \circ i = J \circ du$. Where $\Sigma$ is a Riemann surface with punctures, and $u$ converges to a periodic orbit at each puncture. A Real pseudoholomorphic curve is a pseudoholomorphic curve with Real symmetry, which means there exists an involution $N$ on $\Sigma$, $N^2 = Id, dN \circ i = -i \circ dN$, and $u(N\cdot) = Nu(\cdot)$. Let $\mathcal{M}_{g, m+n}(\alpha_1, \cdots, \alpha_m; \alpha'_1, \cdots, \alpha'_n)$ denote the moduli space of pseudoholomorphic curves with $m$ positive punctures and $n$ negative punctures, which converge to $\alpha_1, \cdots, \alpha_m, \alpha'_1, \cdots, \alpha'_n$; Let $\mathcal{M}^R_{g, k+l,m+n}(\beta_1, \cdots, \beta_k; \beta'_1, \cdots, \beta'_l; \alpha_1, \cdots, \alpha_m; \alpha'_1, \cdots, \alpha'_n)$ denote the moduli space of Real pseudoholomorphic curves with $k$ positive punctures and $l$ negative punctures, where they converge to brake orbits $\beta_1, \cdots, \beta_k; \beta'_1, \cdots, \beta'_l$, and with $m$ pairs of positive punctures and $n$ pairs of negative punctures, where they converge symmetric to periodic orbits $\alpha_1(t), N\alpha_1(-t) \cdots, \alpha_m(t), N\alpha_m(-t), \alpha'_1(t),\\ N\alpha'_1(-t),\cdots, \alpha'_n(t), N\alpha'_n(-t)$ in pairs, see \cite[\S 3]{ZZ} . We abbreviate the notation by $\mathcal{M}$ and $\mathcal{M}^R$, when $g=0$ and it is not necessary to mention the asymptotics.

It is well know that the virtual dimension of the moduli space of pseudoholomorphic curves is given by the following formula\cite[\S 3.3.11 Theorem]{Schwarz}, \cite[Theorem 1.8]{Dragnev}
\begin{align*}
&\text{vir dim}\mathcal{M}_{g, m+n}(\alpha_1, \cdots, \alpha_m; \alpha'_1, \cdots, \alpha'_n) \\
&\quad = -\chi(\Sigma) + 2c_1(u(\Sigma)) + \sum_{p = 1}^{m}\mu_{CZ}(\alpha_p) - \sum_{q = 1}^{n}\mu_{CZ}(\alpha'_{q})
\end{align*}
Where $c_1(u(\Sigma))$ is the first Chern class of $\xi$ by a trivialization along $u(\Sigma)$, $\mu_{CZ}$ are the Conley-Zehnder index for the periodic Reeb orbits defined via the same trivialization.

In \cite{ZZ} the author and C.Zhu calculated that the virtual dimension of the moduli space of Real pseudoholomorphic curves, which is given by
\begin{align*}
&\text{vir dim}\mathcal{M}^R_{g, k+l,m+n}(\beta_1, \cdots, \beta_k; \beta'_1, \cdots, \beta'_l; \alpha_1, \cdots, \alpha_m; \alpha'_1, \cdots, \alpha'_n)\\
&= -\frac{1}{2}\chi(\Sigma) + c_1(u(\Sigma)) + \sum_{i = 1}^{k}\mu_1(\beta_i) - \sum_{j = 1}^{l}\mu_1(\beta'_{j}) +\sum_{p = 1}^{m}\mu_{CZ}(\alpha_i) - \sum_{q = 1}^{n}\mu_{CZ}(\alpha'_{q})
\end{align*}

\subsection{Iteration formulae for brake orbit}

We can iterate a periodic Reeb orbit $(\alpha, \tau)$ and get new periodic Reeb orbits $(\alpha^k, k\tau)$. The same holds for brake orbits $(\beta^k, k\tau)$. It is well known the symplectic path $\Psi_{\alpha}(t)$ of a periodic orbit is iterated from the first period by $\Psi_{\alpha}(t + \tau) = \Psi_{\alpha}(t)\Psi_{\alpha}(\tau)$. The symplectic path $\Psi_{\beta}(t)$ of a brake orbit is iterated after the first half period by $\Psi_{\beta}(t) = N\Psi_{\beta}(\tau - t)\Psi_{\beta}(\frac{\tau}{2})^{-1}N\Psi_{\beta}(\frac{\tau}{2}), t \in [\frac{\tau}{2}, \tau], \Psi_{\beta}(t + \tau) = \Psi_{\beta}(t)\Psi_{\beta}(\tau)$, see \cite[Equation (4.2)]{LZ}.

Periodic Reeb orbits have three types:elliptic, positive and negative hyperbolic, which depend on the eigenvalue of $\Psi_{\alpha}(\tau)$. When $\Psi_{\alpha}(\tau)$ has eigenvalues on the unit circle, $\alpha$ is elliptic. When $\Psi_{\alpha}(\tau)$ has positive or negative real eigenvalues, $\alpha$ is positive or negative hyperbolic. Elliptic orbit has iteration formula $\mu_{CZ}(\alpha^k) = 2[k\theta] + 1$. And a hyperbolic orbit has formula $\mu_{CZ}(\alpha^k) =km$, where $m$ is an odd number, when $\alpha$ is negative hyperbolic, and $m$ is an even number, when $\alpha$ is positive hyperbolic.

D.Zhang and C.Liu have given the abstract precise iteration formulae for brake orbits in all dimensions \cite[Theorem 1.3]{LZ}. Based on their works, we can give the precise iteration formulae in dimension $2$. There are only finitely many cases as well. We will use them to study moduli spaces of Real pseudoholomorphic curves in dimension $3$ later.

For the elliptic case, $\mu_1(\beta) = \frac{1}{2}\mu_{CZ}(\beta)$. For the hyperbolic case, both the positive and the negative case have two subcases. We can distinguish them by the sign of $\mu_1(\beta^2) - \mu_2(\beta^2)$. Therefore we call a brake orbit $\beta$ type one, if $\mu_1(\beta^2) - \mu_2(\beta^2) > 0$. And we call a brake orbit $\beta$ type two if $\mu_1(\beta^2) - \mu_2(\beta^2) < 0$. So we can list all patterns.

\begin{theo}
Any nondegenerate brake orbit belongs to one of the following five cases
\end{theo}

$\beta$ elliptic, $\beta = R(2\pi\theta')$
$$R(\theta') = \begin{pmatrix}
\cos\theta' & -\sin\theta'\\
\sin\theta' & \cos\theta'
\end{pmatrix}$$
$$\mu_1(\beta^k) = [k\theta] + \frac{1}{2}, \mu_{CZ}(\beta^k) = 2[k\theta] + 1$$

$\beta$ negative hyperbolic type one, $\mu_1(\beta^2) - \mu_2(\beta^2) = 1$, $\mu_{CZ}(\beta) = 2\mu_1(\beta)$
\begin{align*}
&\mu_1(\beta^k) = \begin{cases}k\mu_1(\beta), &k\ \text{is odd}\\
  k\mu_1(\beta) + \frac{1}{2}, &k\ \text{is even}
\end{cases}\\
&\mu_{CZ}(\beta^k) = k\mu_{CZ}(\beta) = 2k\mu_1(\beta)
\end{align*}

$\beta$ negative hyperbolic type two, $\mu_1(\beta^2) - \mu_2(\beta^2) = -1$, $\mu_{CZ}(\beta) = 2\mu_1(\beta)$
\begin{align*}
&\mu_1(\beta^k) = \begin{cases}k\mu_1(\beta), &k\ \text{is odd}\\
k\mu_1(\beta) - \frac{1}{2}, &k\ \text{is even}
\end{cases}\\
&\mu_{CZ}(\beta^k) = k\mu_{CZ}(\beta) = 2k\mu_1(\beta)
\end{align*}

$\beta$ positive hyperbolic type one $\mu_1(\beta^2) - \mu_2(\beta^2) = 1, \mu_{CZ}(\beta) = 2\mu_1(\beta) - 1$
\begin{align*}
&\mu_1(\beta^k) = k\mu_1(\beta) + \frac{1-k}{2}\\
&\mu_{CZ}(\beta^k) = k\mu_{CZ}(\beta) = 2k\mu_1(\beta) - k
\end{align*}

$\beta$ positive hyperbolic type two $\mu_1(\beta^2) - \mu_2(\beta^2) = -1, \mu_{CZ}(\beta) = 2\mu_1(\beta) + 1$
\begin{align*}
&\mu_1(\beta^k) = k\mu_1(\beta) + \frac{k-1}{2}\\
&\mu_{CZ}(\beta^k) = k\mu_{CZ}(\beta) = 2k\mu_1(\beta) + k
\end{align*}

For convenience of computation, we give the GIT description of all cases in the Appendix B \ref{Appendix 2}.

\subsection{Index Inequalities for Real pseudoholomorphic curves}
Following an idea of Michael Hutchings in \cite{Hutchings}\cite{HN}, we can get the counterpart of the index inequalities for moduli spaces of Real pseudoholomorphic curves, which are important to establish Real ECH and Real cylindrical contact homology in dimension $3$ in the future.

We define a Real pseudoholomorphic curve $u: \Sigma \to {\bf R} \times Y$ to be a Real branched cover, if there is another Real pseudoholomorphic curve $u': \Sigma' \to {\bf R} \times Y$, such that $u = u'\circ g$ and $g: \Sigma \to \Sigma'$ is a Real pseudoholomorphic function $g(N\cdot) = Ng(\cdot)$. We also call branched covers as multiple covers. If $g(x) = x'$, and $g(z)= z^k$, in local coordinate of $x, x'$, such that $x, x'$ are original point of ${\bf C}$. $x$ is called a branch point and $k-1$ the branch number of $g$ at $x$ which is denoted by $b_g(x)$. The total branch number of $g$ is defined to be the sum of branch number $B = \sum_{x \in \Sigma}(b_g(x))$. We define the covering multiplicity and total branch number of $u$ over $u'$ to be the degree of $g$ and the total branch number of $g$. See Chapter \uppercase\expandafter{\romannumeral 1} \S 1 and \S 2 in \cite{Kra}.

Our main results are the following three inequalities:

\begin{theo}(Real ECH lemma)
If $u$ is a Real pseudoholomorphic curve in ${\bf R} \times Y$ which is a Real branched cover of a Real trivial cylinder ${\bf R} \times \beta$, where $\beta$ is a brake orbit, then $ind_R(u) \geq 0$.
\end{theo}

We will define the Real ECH generator and the Real ECH index in section $3$. Let $\beta_1, \beta_2$ be Real ECH generators, a flow line $u$ from $\beta_1$ to $\beta_2$ is a Real pseudoholomorphic curve which converges to $\beta_1, \beta_2$ at positive and negative punctures, and its projection to $Y$ represents a class $Z \in H_2(Y; \beta_1, \beta_2)$.
\begin{theo}(Real ECH inequality) Let $\beta_1, \beta_2$ be Real ECH generators, $u$ is a flow fine from $\beta_1$ to $\beta_2$, then the Real Fredholm index is smaller than or equals the Real ECH index, $ind_R u \leq I_{RECH}(\beta_1, \beta_2; Z)$.
Equality holds only if the ends of $u$ satisfies an unique partition which we will give in Theorem \ref{Par}.
\end{theo}

\begin{theo}(Index inequality of Real multiple cover)
If $u$ is a Real pseudoholomorphic curve  which is a Real multiple cover of a somewhere injective Real pseudoholomorphic curve $\bar u$, let $D$ denote the covering multiplicity and $B$ denote the total branch number of $u$ over $u'$, then $ind_R u \geq Dind_R(\bar u) + (B + 1 + D) - \#_1 - \#_2$
\end{theo}
Where $\#_1$ is the number of hyperbolic positive type one pairs which cover a brake orbit at positive asymptotics; $\#_2$ is the number of hyperbolic positive type two pairs which cover a brake orbit at negative asymptotics.

In Appendix A \ref{Appendix 1} we will give the possible application of those inequalities to the construction of Real cylindrical contact homology.

\section{Iteration Formulae}
In \cite{LZZ}\cite{LZ}, the authors defined indices $\mu_1, \mu_2, i_{L_0}, i_{L_1}, i_{\sqrt{-1}}^{L_0}, i_{\sqrt{-1}}^{L_1},$ and they satisfy following equations
\begin{align}
&\mu_1(\beta)= \frac{1}{2} + i_{L_0}(\beta), \mu_2(\beta)= \frac{1}{2} + i_{L_1}(\beta)\\
&\mu_1(\beta^2)= \frac{1}{2} + i_{L_0}(\beta)+i_{\sqrt{-1}}^{L_0}(\beta), \mu_2(\beta^2)= \frac{1}{2} + i_{L_1}(\beta) + i_{\sqrt{-1}}^{L_1}(\beta)
\end{align}
They also established the following theorems, \cite[Theorem 1.3, Theorem 4.2]{LZ}
\begin{theo}\label{LZ}
	Let $\beta$ be a brake orbit and $\Psi_{\beta}(t)$ be the corresponding symplectic path which is iterated after the first half period $\Psi_{\beta}(t) = N\Psi_{\beta}(\tau - t)\Psi_{\beta}(\frac{\tau}{2})^{-1}N\Psi_{\beta}(\frac{\tau}{2}), t \in [\frac{\tau}{2}, \tau], \Psi_{\beta}(t + \tau) = \Psi_{\beta}(t)\Psi_{\beta}(\tau)$.

When $k$ is odd, there holds
\begin{equation}
i_{L_0}(\beta^k) = i_{L_0}(\beta) + \sum_{i=1}^{\frac{k-1}{2}}i_{\omega_k^{2i}}(\beta)
\end{equation}

When $k$ is even, there holds
\begin{align}
i_{L_0}(\beta^k) = i_{L_0}(\beta) + i^{L_0}_{\sqrt{-1}}(\beta) + \sum_{i=1}^{\frac{k}{2}-1}i_{\omega_k^{2i}}(\beta)
\end{align}
where $\omega_k = e^{\pi\sqrt{-1}/ k}$, and $i_{\omega}$ is the index function introduced in \cite[page 130 Definition 3]{Long}.
\end{theo}

\begin{theo}\label{LZ1}
	Let $\beta$ be a brake orbit and $\Psi_{\beta}(t)$ be the corresponding symplectic path which is iterated after the first half period by $\Psi_{\beta}(t) = N\Psi_{\beta}(\tau - t)\Psi_{\beta}(\frac{\tau}{2})^{-1}N\Psi_{\beta}(\frac{\tau}{2}), t \in [\frac{\tau}{2}, \tau], \Psi_{\beta}(t + \tau) = \Psi_{\beta}(t)\Psi_{\beta}(\tau)$.

When $k$ is odd, there holds
\begin{equation}
i_{L_1}(\beta^k) = i_{L_1}(\beta) + \sum_{i=1}^{\frac{k-1}{2}}i_{\omega_k^{2i}}(\beta)
\end{equation}

When $k$ is even, there holds
\begin{align}
i_{L_1}(\beta^k) = i_{L_1}(\beta) + i^{L_1}_{\sqrt{-1}}(\beta) + \sum_{i=1}^{\frac{k}{2}-1}i_{\omega_k^{2i}}(\beta)
\end{align}
where $\omega_k = e^{\pi\sqrt{-1}/ k}$, and $i_{\omega}$ is the index function introduced in \cite[page 130 Definition 3]{Long}.
\end{theo}

For a contact manifold in dimension $3$, we have $\Psi_\beta(t) \in \mathrm{Sp}(2)$. We assume $\beta$ to be nondegenerate, so $\beta$ is a hyperbolic or an irrational elliptic periodic orbit.

When $\beta$ is hyperbolic with period $\tau$, we have $i_\omega(\beta) = i(\beta) = \mu_{CZ}(\beta)$, for any $\omega = e^{i\theta}$. Because we have the following formula
$$i_\omega(\beta) = i(\beta) + \Sigma_{0 \leq \theta < \theta_0}S^+_{\Phi_\beta(\tau)} - \Sigma_{0 < \theta \leq \theta_0}S^-_{\Phi_\beta(\tau)}$$
where $\omega = e^{i\theta}$, and $S^{\pm}_M$ is the splitting number, which is defined in \cite[\S 9.1 Definition 4]{Long}. For the above formula see \cite[Equation (3)]{Long}. The splitting number $S_{\Phi_\beta(\tau)}$ is $0$, by our hyperbolic assumption, see \cite[page 199, List 12 $\langle8\rangle$]{Long}. And our result follows.

By the above theorems \ref{LZ} \ref{LZ1} we have,
\begin{equation}
i_{L_0}(\beta^k) =
\begin{cases}
i_{L_0}(\beta) + \frac{k-1}{2}\mu_{CZ}(\beta), &k\ \text{is odd}; \\
i_{L_0}(\beta) + i^{L_0}_{\sqrt{-1}}(\beta) + (\frac{k}{2}-1)\mu_{CZ}(\beta), &k\ \text{is even}.
\end{cases}
\end{equation}

We consider the relation between $\mu_1, \mu_2$ and $\mu_{CZ}$. From Long, Zhang, Zhu's paper \cite[Proposition C]{LZZ}, we know $\mu_{CZ}(\beta^k) = \mu_1(\beta^k) + \mu_2(\beta^k)$. Therefore from Theorem \ref{LZ}, we have
\begin{align}
&\mu_1(\beta)= \frac{1}{2} + i_{L_0}(\beta), \mu_2(\beta)= \frac{1}{2} + i_{L_1}(\beta)\\
&\mu_1(\beta^2)= \frac{1}{2} + i_{L_0}(\beta)+i_{\sqrt{-1}}^{L_0}(\beta), \mu_2(\beta^2)= \frac{1}{2} + i_{L_1}(\beta) + i_{\sqrt{-1}}^{L_1}(\beta)\\
&\mu_{CZ}(\beta) = 1 + i_{L_0}(\beta) + i_{L_1}(\beta), \\
&\mu_{CZ}(\beta^2) = 1 + i_{L_0}(\beta) + i_{\sqrt{-1}}^{L_0}(\beta) + i_{L_1}(\beta) + i_{\sqrt{-1}}^{L_1}(\beta)
\end{align}

From \cite[Theorem 2.3 and Equation (3.21)]{LZ}, we know the inequalities $|\mu_1(\beta^k) - \mu_2(\beta^k)| \leq 1$ and $i_{L_0}(\beta) \leq i_{\sqrt{-1}}^{L_0}(\beta) \leq i_{L_0}(\beta) + 1$, $i_{L_1}(\beta) \leq i_{\sqrt{-1}}^{L_1}(\beta) \leq i_{L_1}(\beta) + 1$.

When $\beta$ is hyperbolic, $\mu_{CZ}(\beta)$ can be an odd or even number. Considering the case $\mu_{CZ}(\beta)$ is an odd number, $i_{L_0}$ and $i_{L_1}$ must be equal, because $i_{L_0}(\beta)$ and $i_{L_1}(\beta)$ can not be an odd and a even number, and the relation $|i_{L_0} - i_{L_1}| \leq 1$. Because $\mu_{CZ}(\beta^2) = 2\mu_{CZ}(\beta) = 2(1 + i_{L_0}(\beta) + i_{L_1}(\beta))$, we have
\begin{align*}
1 + i_{L_0}(\beta) + i_{\sqrt{-1}}^{L_0}(\beta) + i_{L_1}(\beta) + i_{\sqrt{-1}}^{L_1}(\beta) &= 2(1 + i_{L_0}(\beta) + i_{L_1}(\beta))\\
i_{\sqrt{-1}}^{L_0}(\beta) + i_{\sqrt{-1}}^{L_1}(\beta) &= 1 + i_{L_0}(\beta) + i_{L_1}(\beta)
\end{align*}
Then we get
\begin{equation}\label{11}
{\sqrt{-1}}^{L_0}(\beta) = i_{L_0}(\beta) +1, i_{\sqrt{-1}}^{L_1}(\beta) = i_{L_0}(\beta)
\end{equation}
or
\begin{equation}\label{12}
i_{\sqrt{-1}}^{L_0}(\beta) = i_{L_0}(\beta), i_{\sqrt{-1}}^{L_1}(\beta) = i_{L_0}(\beta) +1
\end{equation}

Using the same method, if $\mu_{CZ}(\beta)$ is an even number, we know
\begin{equation}\label{13}
i_{L_1}(\beta) = i_{L_0}(\beta) \pm 1
\end{equation}
\begin{equation}\label{14}
i_{\sqrt{-1}}^{L_0}(\beta) = i_{\sqrt{-1}}^{L_1}(\beta)
\end{equation}

And when $\beta$ is elliptic, $\mu_{CZ}$ is an odd number, by the same reason as in hyperbolic negative, we always have $\mu_1(\beta^k) = \frac{1}{2}\mu_{CZ}(\beta^k)$. After elementary analysis, we can list all cases as follows:

When $\beta$ is elliptic, $\beta = R(\theta')$. Since $\mu_{CZ}(\beta^{k}) = 2[k\theta] + 1$ is always odd, we know $\mu_1(\beta^k) = \mu_2(\beta^k) = [k\theta] + \frac{1}{2}$.

When $\beta$ is hyperbolic, both positive or negative hyperbolic cases have two subcases according to the sign of $\mu_1(\beta^2) - \mu_2(\beta^2)$. We call the case $\mu_1(\beta^2) - \mu_2(\beta^2) > 0$ type one; the case $\mu_1(\beta^2) - \mu_2(\beta^2) < 0$ type two. The difference $\mu_1(\beta) - \mu_2(\beta)$ is given by the H\"{o}rmander index in the proof of \cite[Theorem 3.3]{LZZ}. Using equations (\ref{11}) (\ref{12}), we have:

Negative hyperbolic type one
\begin{align}
&i_{L_0}(\beta), i_{L_1}(\beta) = i_{L_0}(\beta)\\
&i_{\sqrt{-1}}^{L_0}(\beta) = i_{L_0}(\beta) + 1, i_{\sqrt{-1}}^{L_1}(\beta) = i_{L_0}(\beta)\\
&\mu_{CZ}(\beta) = 1 + 2i_{L_0}(\beta)\\
& \mu_1(\beta) = \frac{1}{2} + i_{L_0}(\beta), \mu_1(\beta^2) = \frac{3}{2} + 2i_{L_0}(\beta)\\
&\mu_1(\beta^k) = \begin{cases}\frac{1}{2} + i_{L_0}(\beta) + \frac{k-1}{2}(1 + 2i_{L_0}(\beta)) = k\mu_1(\beta), &k\ \text{is odd}\\
\frac{1}{2} + i_{L_0}(\beta) + i_{L_0}(\beta) + 1 + (\frac{k}{2}-1)(1 + 2i_{L_0}(\beta)) = k\mu_1(\beta) + \frac{1}{2}, &k\ \text{is even}
	\end{cases}\label{hn1}
\end{align}

Negative hyperbolic type two
\begin{align}
&i_{L_0}(\beta), i_{L_1}(\beta) = i_{L_0}(\beta)\\
&i_{\sqrt{-1}}^{L_0}(\beta) = i_{L_0}(\beta), i_{\sqrt{-1}}^{L_1}(\beta) = i_{L_0}(\beta) + 1\\
&\mu_{CZ}(\beta) = 1 + 2i_{L_0}(\beta)\\
& \mu_1(\beta) = \frac{1}{2} + i_{L_0}(\beta), \mu_1(\beta^2) = \frac{1}{2} + 2i_{L_0}(\beta)\\
&\mu_1(\beta^k) = \begin{cases}\frac{1}{2} + i_{L_0}(\beta) + \frac{k-1}{2}(1 + 2i_{L_0}(\beta)) = k\mu_1(\beta), &k\ \text{is odd}\\
\frac{1}{2} + i_{L_0}(\beta) + i_{L_0}(\beta) + (\frac{k}{2}-1)(1 + 2i_{L_0}(\beta)) = k\mu_1(\beta) - \frac{1}{2}, &k\ \text{is even}
\end{cases}\label{hn2}
\end{align}

Using equations (\ref{13}), (\ref{14}) we have:

Positive hyperbolic type one
\begin{align}
&i_{L_0}(\beta), i_{L_1}(\beta) = i_{L_0}(\beta) - 1\\
&i_{\sqrt{-1}}^{L_0}(\beta) = i_{L_0}(\beta), i_{\sqrt{-1}}^{L_1}(\beta) = i_{L_0}(\beta)\\
&\mu_{CZ}(\beta) = 2i_{L_0}(\beta)\\
& \mu_1(\beta) = \frac{1}{2} + i_{L_0}(\beta), \mu_1(\beta^2) = \frac{1}{2} + 2i_{L_0}(\beta)\\
&\mu_1(\beta^k) = \begin{cases}\frac{1}{2} + i_{L_0}(\beta) + \frac{k-1}{2}(2i_{L_0}(\beta)) = k\mu_1(\beta) + \frac{1-k}{2}, &k\ \text{is odd}\\
\frac{1}{2} + i_{L_0}(\beta) + i_{L_0}(\beta) + (\frac{k}{2}-1)(2i_{L_0}(\beta)) = k\mu_1(\beta) + \frac{1-k}{2}, &k\ \text{is even}
\end{cases}\\
&\mu_1(\beta^k) = k\mu_1(\beta) + \frac{1-k}{2}\label{hp1}
\end{align}

Positive hyperbolic type two
\begin{align}
&i_{L_0}(\beta), i_{L_1}(\beta) = i_{L_0}(\beta) + 1\\
&i_{\sqrt{-1}}^{L_0}(\beta) = i_{L_0}(\beta) + 1, i_{\sqrt{-1}}^{L_1}(\beta) = i_{L_0}(\beta) + 1\\
&\mu_{CZ}(\beta) = 2 + 2i_{L_0}(\beta)\\
& \mu_1(\beta) = \frac{1}{2} + i_{L_0}(\beta), \mu_1(\beta^2) = \frac{3}{2} + 2i_{L_0}(\beta)\\
&\mu_1(\beta^k) = \begin{cases}\frac{1}{2} + i_{L_0}(\beta) + \frac{k-1}{2}(2 + 2i_{L_0}(\beta)) = k\mu_1(\beta) + \frac{k - 1}{2}, &k\ \text{is odd}\\
\frac{1}{2} + i_{L_0}(\beta) + i_{L_0}(\beta) +1 + (\frac{k}{2}-1)(2+2i_{L_0}(\beta)) = k\mu_1(\beta) + \frac{k - 1}{2}, &k\ \text{is even}
\end{cases}\\
&\mu_1(\beta^k) = k\mu_1(\beta) + \frac{k - 1}{2}\label{hp2}
\end{align}

\section{Real ECH lemma and Real ECH inequality}

In \cite{HT1}, a fundamental lemma for ECH is \cite[Lemma 1.7]{HT1}
\begin{lemma}
	If $u$ is a pseudoholomorphic curve in ${\bf R} \times Y$ which is a branched cover of a trivial cylinder, then $ind(u) \geq 0$.
\end{lemma}

Here we give the Real counterpart of this lemma,
\begin{theo}\label{RECHLE}
If $u$ is a Real pseudoholomorphic curve in ${\bf R} \times Y$ which is a Real branched cover of a Real trivial cylinder ${\bf R} \times \beta$, where $\beta$ is a brake orbit, then $ind_R(u) \geq 0$.
\end{theo}

\begin{proof}

Let $u$ be a Real pseudoholomorphic curve which has $k$ positive symmetric punctures, $l$ negative symmetric punctures, $m$ pairs of nonsymmetric positive punctures, and $n$ pairs of nonsymmetric negative punctures. The punctures are of covering multiplicities $a_1, \ldots, a_{k'}, a_{k'+1}, \ldots, a_{k}; b_1, \ldots, b_{l'}, b_{l'+1}, \ldots, b_{l}; c_1, \ldots, c_{m'},\\ c_{m'+1}, \ldots, c_{m}; d_1, \ldots, d_{n'}, d_{n'+1}, \ldots, d_{n}$. At the same time $a_1, \ldots, a_{k'}; b_1, \ldots, b_{l'};\\ c_1, \ldots, c_{m'}; d_1, \ldots, d_{n'}$ are odd, $a_{k'+1}, \ldots, a_{k}; b_{l'+1}, \ldots, b_{l};  c_{m'+1}, \ldots, c_{m};\\  d_{n'+1}, \ldots, d_{n}$ are even. They satisfy the relation
\begin{equation}\label{sum}
a_1 + \cdots a_k + 2c_1 + \cdots 2c_m = b_1 + \cdots b_l + 2d_1 + \cdots + 2d_n
\end{equation}
Thus from \cite[Main Theorem]{ZZ}, we have the result,

\begin{multline}
\text{ind}_R(u) = -\frac{1}{2}(2-2g-k-l-2m-2n) + \sum_{i = 1}^{k}\mu_1(\beta^{a_i}) - \sum_{j = 1}^{l}\mu_1(\beta^{b_j}) \\+ \sum_{p = 1}^{m}\mu_{CZ}(\beta^{c_p}) - \sum_{q = 1}^{n}\mu_{CZ}(\beta^{d_q})
\end{multline}

Thanks to the index iteration formulae, we have:

If $\beta$ is hyperbolic,
\begin{align*}
&\text{ind}_R(u)= -\frac{1}{2}(2-2g-k-l-2m-2n) \\
&+ \sum_{i=1}^{k'}(\mu_1(\beta) + \frac{a_i-1}{2}\mu_{CZ}(\beta)) + \sum_{i=k'+1}^{k}(\mu_1(\beta^2) + (\frac{a_i}{2}-1)\mu_{CZ}(\beta))\\
&- \sum_{j=1}^{l'}(\mu_1(\beta) + \frac{b_j-1}{2}\mu_{CZ}(\beta)) - \sum_{j=l'+1}^{l}(\mu_1(\beta^2) + (\frac{b_j}{2}-1)\mu_{CZ}(\beta))\\
&+ \sum_{p = 1}^{m}c_p\mu_{CZ}(\beta) - \sum_{q = 1}^{n}d_q\mu_{CZ}(\beta)\\
&= -\frac{1}{2}(2-2g-k-l-2m-2n)\\
& +k'(\mu_1(\beta) - \frac{\mu_{CZ}(\beta)}{2}) + (k - k')(\mu_1(\beta^2) - \mu_{CZ}(\beta))\\
& -l'(\mu_1(\beta) - \frac{\mu_{CZ}(\beta)}{2}) - (l - l')(\mu_1(\beta^2) - \mu_{CZ}(\beta))\\
&+(\sum_{i=1}^{k}\frac{a_i}{2} - \sum_{j=1}^{k}\frac{b_j}{2} + \sum_{p=1}^{m}c_p - \sum_{q=1}^{n}d_q)\mu_{CZ}(\beta)\\
& =\frac{1}{2}(2g+k+l+2m+2n-2) + (k'-l')(\mu_1(\beta) - \frac{\mu_{CZ}(\beta)}{2})\\
& +((k-k')-(l-l'))(\mu_1(\beta) - \mu_{CZ}(\beta))
\end{align*}
In the last equality, we used the relation of equation (\ref{sum}).

Let $\epsilon_1 = \mu_1(\beta) - \frac{\mu_{CZ}(\beta)}{2}$, $\epsilon_2 = \mu_1(\beta^2) - \mu_{CZ}(\beta)$. From our list of iteration formulae, we have:

\begin{enumerate}
	\item hyperbolic type one $\epsilon_1 = 0, \epsilon_2 = \frac{1}{2}$, $ind_R(u) = \frac{1}{2}(2g+k+l+2m+2n-2 + (k-k') - (l-l'))$
	\item hyperbolic negative type two $\epsilon_1 = 0, \epsilon_2 = -\frac{1}{2}$, $ind_R(u) = \frac{1}{2}(2g+k+l+2m+2n-2 - (k-k') + (l-l'))$
	\item hyperbolic positive type one $\epsilon_1 = \frac{1}{2}, \epsilon_2 = \frac{1}{2}$, $ind_R(u) = \frac{1}{2}(2g+k+l+2m+2n-2 + k - l)$
	\item hyperbolic positive type two $\epsilon_1 = -\frac{1}{2}, \epsilon_2 = -\frac{1}{2}$, $ind_R(u) = \frac{1}{2}(2g+k+l+2m+2n-2 - k + l)$
\end{enumerate}

We have $k \geq 1, l \geq 1, 0 \leq k' \leq k, 0 \leq l' \leq l$,

In case negative hyperbolic type one, the minimum of right formula $\frac{1}{2}(2g+2k+2m+2n-2-k'+l')$ is $-\frac{1}{2}$, only if $g=0, m=0, n=0, k=k'=1, l'=0$. But $a_1, \ldots, a_{k'}$ are odd numbers, and $a_{k'+1}, \ldots, a_{k}$ are even numbers, therefore we have $a_1 + \cdots a_k + 2c_1 + \cdots 2c_m \equiv k' \text{mod} 2$, from the relation of (\ref{sum}), we get $k' \equiv l'\ \text{mod} 2$, which contradict the minimum case. So $ind_R(u) \geq 0$. It is easy to see, $ind_R (u) = 0$ if and only if $k = l = 1$ or $k=k'=2, l' = 0$. Case negative hyperbolic type two is similar.

In case Positive hyperbolic type one, the minimum of $ind_R(u)$ is $0$, when $g=0, m=0, n=0, k=1$. Case Positive hyperbolic type two is similar.

If $\beta$ is elliptic, $\Psi_{\beta}(\tau) = R(\theta)$, $\mu_1(\beta^k) = \frac{1}{2}\mu_{CZ}(\beta^k)$
\begin{align*}
&\text{ind}_R(u)= -\frac{1}{2}(2-2g-k-l-2m-2n) \\
&+ \sum_{i=1}^{k}\frac{1}{2}\mu_{CZ}(\beta^{a_i})- \sum_{j=1}^{l}\frac{1}{2}\mu_{CZ}(\beta^{b_j}) \\
&+ \sum_{p = 1}^{m}\mu_{CZ}(\beta^{c_p}) - \sum_{q = 1}^{n}\mu_{CZ}(\beta^{d_q})\\
&= -\frac{1}{2}(2-2g-k-l-2m-2n)\\
&+ \sum_{i = 1}^{k}\lfloor a_i\theta\rfloor + \frac{k}{2} - \sum_{j = 1}^{l}\lfloor b_j\theta\rfloor - \frac{l}{2} + \sum_{p = 1}^{m}2\lfloor c_p\theta\rfloor + m - \sum_{q = 1}^{n}2\lfloor d_q\theta\rfloor - n\\
& =\frac{1}{2}(2g-2) + \sum_{i = 1}^{k}\lceil a_i\theta\rceil - \sum_{j = 1}^{l}\lfloor b_j\theta\rfloor + \sum_{p = 1}^{m}2\lceil c_p\theta\rceil - \sum_{q = 1}^{n}2\lfloor d_q\theta\rfloor
\end{align*}

We define $\text{ind}_\theta(a_1, \ldots, a_k; b_1, \ldots, b_l; c_1, \ldots, c_m; d_1, \ldots, d_n) := \sum_{i = 1}^{k}\lceil a_i\theta\rceil - \sum_{j = 1}^{l}\lfloor b_j\theta\rfloor + \sum_{p = 1}^{m}2\lceil c_p\theta\rceil - \sum_{q = 1}^{n}2\lfloor d_q\theta\rfloor -1$. Let $M = a_1 + \cdots a_k + 2c_1 + \cdots 2c_m = b_1 + \cdots b_l + 2d_1 + \cdots + 2d_n$. $\sum_{i = 1}^{k}\lceil a_i\theta\rceil + \sum_{p = 1}^{m}2\lceil c_p\theta\rceil \geq \lceil M\theta \rceil$, $\lfloor M\theta\rfloor \geq \sum_{j = 1}^{l}\lfloor b_j\theta\rfloor + \sum_{q = 1}^{n}2\lfloor d_q\theta\rfloor$. That yields $\text{ind}_\theta \geq 0$. And the equality holds, if and only if $\sum_{i = 1}^{k}\lceil a_i\theta\rceil + \sum_{p = 1}^{m}2\lceil c_p\theta\rceil = \lceil \frac{M\theta}{2\pi} \rceil$ and $\lfloor M\theta\rfloor = \sum_{j = 1}^{l}\lfloor b_j\theta\rfloor + \sum_{q = 1}^{n}2\lfloor d_q\theta\rfloor$.

\end{proof}

Remark: The elliptic case can also be proved by the relation $ind_R u = \frac{1}{2}ind u$ and the result in nonsymmetric case.

We can define Real ECH index and use the iteration formulae for brake orbits to get the real partition of real pseudoholomorphic curves by means of the index inequalities like in \cite[Theorem 1.7]{Hutchings}, which is crucial to define Real ECH.

Similar to the generator of ECH, we define
\begin{defi}
	A Real generator is a finite set of pairs $\beta = \{(\beta_1, m_1) \ldots, (\beta_k, m_k)\}$, where $\beta_1, \ldots, \beta_k$ are disjoint brake orbit and $m_1, \ldots, m_k \in {\bf N}^+$(multiplicities).
	
	a flow line $u$ from $\alpha$ to $\beta$ is a Real pseudoholomorphic curve which converges to $\alpha, \beta$ at positive and negative punctures, and its projection to $Y$ represents a class $Z \in H_2(Y; \alpha, \beta)$.
\end{defi}

If the partition associated to each orbit $\beta_i$ is $(q_{i, 1}, q_{i, 2}, \ldots)$, from the paper of the author and C.Zhu\cite[Main Theorem]{ZZ}, we know the Fredholm index of  Real pseudoholomorphic curve is
\begin{equation}
\text{ind}_{\text R}(u) = -\frac{1}{2}\chi(\Sigma) + c_1(u) + \sum_{i}\sum_{r}\mu_1(\alpha_i^{q_{i, r}}) - \sum_{j}\sum_{r}\mu_1(\alpha_j^{q_{j, r}})
\end{equation}

Comparing the Fredholm index of Real pseudoholomorphic curves, we should define the Real ECH index as the following
\begin{defi}
We define the Real ECH index
\begin{equation*}
I_{\text RECH}(\alpha, \beta; Z) = \frac{1}{2}c_1(u) + \frac{1}{2}Q(Z) + \sum_{i}\sum_{k=1}^{m_i}\mu_1(\alpha_i^k) - \sum_{j}\sum_{k=1}^{n_j}\mu_1(\beta_j^k)
\end{equation*}
\end{defi}
Where $Q(Z)$ is the self intersection number.

And this Real ECH index is really the upper bound of Fredholm index of Real pseudoholomorphic curves for all partition of ends. At the same time, the upper bound is reached by a unique partition.

Here is the theorem,
\begin{theo}\label{Par}
$\text{ind}_{\text R}(u) \leq I_{\text RECH}(u)$, and equality holds if and only if the ends satisfy a unique partition. Without loss of generality, we consider one brake orbit with total multiplicity $(\beta, n)$. The equality holds only if the negative ends satisfy the partition
\begin{enumerate}
\item $\alpha$ is elliptic, the same partition as in ECH, see \cite[\S 4]{Hutchings}, which is determined by the parameter $\theta, \Psi_\beta(\tau) = R(\theta)$
\item $\beta$ is positive hyperbolic,
\begin{enumerate}[i.]
\item positive hyperbolic type two, $(1,\ldots, 1)$, the same partition as in ECH,
\item positive hyperbolic type one, $(n)$,
\end{enumerate}
\item $\beta$ is negative hyperbolic,
\begin{enumerate}[i.]
\item negative hyperbolic type two, $(2,\ldots, 2)$ or $(2,\ldots, 2, 1)$, the same partition as in ECH,
\item negative hyperbolic type one, $(n)$, when $n$ is odd; $(1, n-1)$, when $n$ is even.
\end{enumerate}
\end{enumerate}
\end{theo}
And for the positive ends, we can reverse the ends by $\tilde{u} = u(-s, -t)$. Therefore positive ends of elliptic orbits satisfy the partition of negative elliptic ends determined by parameter $-\theta$, which is the same as ECH. The reverse process does not change the positive or negative hyperbolic property, but it switches the type one and two, because of $\Psi_{\beta}(-\tau) = \Psi_{\beta}(\tau)^{-1}$ see Appendix 2 equation (\ref{Appendix}). Hence the positive ends of hyperbolic positive type one orbits satisfy the partition of negative ends of hyperpobic positive type two and verse visa. The same holds for the hyperbolic negative orbits.

Without loss of generality, we just need to consider the situation of one orbit with total multiplicity $(\beta, n)$ at negative punctures. Assume $u$ is convergent to a brake orbit $\beta$ (maybe with multiplicity) at a negative puncture.

Based on Hofer, Wysocki and Zehnders' result\cite{Hofer1}, \cite{Hofer2} and  Frauenfelder and Kangs' result \cite{Urs}. We can choose a neighborhood $E$ of $\beta$ such that $E \cong S^1 \times {\bf R}^2$. The global linearised operator $\partial_s + J \partial_t$ along $u(s, t)$ can be expressed by $\partial_s + J_0 \partial_t + S(s, t)$ in this neighborhood, where $S$ is a symmetric matrix which satisfies $S(s, -t) = N_0S(s, t)N_0$.

Let ${\widetilde W}^{1, 2} := \{x \in W^{1, 2}(S^1, {\bf R}^2)|  x(-t) = Nx(t)\}, {\widetilde L}^2 := \{x \in L^{1, 2}(S^1, {\bf R}^2)| \\ x(-t) = Nx(t)\}$, $A = -J_0 \partial_t - S(s, t): {\widetilde W}^{1, 2} \subset {\widetilde L}^2 \to {\widetilde L}^2$. From the nondegenerate condition, we know that the kernel of operator $A$ is $\{0\}$, and there is a countable set of eigenfunction $e_n$ with $Ae_n = \lambda_n e_n$, which constitute an orthonormal basis for ${\widetilde L}^2$. $u$ can be expressed by $\{e_n(t)\}$, when $s \ll 0$. The following two lemmas are proved in \cite[\S 3, \S 4]{Urs}

\begin{lemma}\label{le1}
	Let $u$ be a Real pseudoholomorphic curve which converges to a brake orbit in $-\infty$, then if $s \ll 0$
\begin{equation*}
	u(s, t) = \sum_{n}a_n e^{\lambda_n s} e_n(t)
\end{equation*}
	with $\lambda_n > 0, e_n(-t) = N e_n(t)$.
\end{lemma}

The winding number $\eta(e)$ for a eigenfunction $e(t)$ is the rotation number of $e(t)$ in ${\bf R}^2$ around the origin.

\begin{lemma}\label{le2}
\begin{enumerate}
	\item If $\lambda_1 < \lambda_2$, $e_1, e_2$ are eigenfunctions corresponding to $\lambda_1, \lambda_2$, then $\eta(e_1) < \eta(e_2)$.
	\item For each integer $\eta$, the space of eigenfunction with winding number $\eta$ is one dimensional.
	\item The maximal winding number for a negative eigenvalue is $\mu_1(\gamma) - \frac{1}{2}$, and the minimal winding number for a positive eigenvalue is $\mu_1(\gamma) + \frac{1}{2}$.
\end{enumerate}

\end{lemma}

Remark: in \cite[\S 3]{Urs}, they defined two winding numbers, the winding number $(u)$ of total period $\tau$ and relative winding number $(u_I)$ of half period $\frac{\tau}{2}$, wind$(u) = 2$ wind$(u_I)$(see \cite[Proposition 4.6]{Urs}), where $u$ is a plane converging to a brake orbit. In our paper, we only use wind$(u)$, the same as in \cite[equation (37)]{Hofer2}.

Let $\xi$ be the braid corresponding to a negative end, supposing $s \ll 0$. The writhe of a braid $\omega(\xi)$ is defined in \cite[\S 3.1]{Hutchings} as the signed number of crossing of the braid $\xi$ in the neighborhood $E$, where counterclockwise twists contribute positively. We can get the bound of the writhe, imitating the method of Hutchings in \cite[Lemma 6.7]{Hutchings},
\begin{lemma}\label{wind}
	Suppose the multiplicity of $\xi$ is $n$, then $\omega(\xi) \geq (n-1)(\mu_1(\beta^n) + \frac{1}{2}),$
	if the equality holds then
\begin{enumerate}
\item If $\beta$ is hyperbolic positive
\begin{enumerate}[i.]
\item positive hyperbolic type one, $n$ is arbitrary
\item positive hyperbolic type two, $n$ is $1$
\end{enumerate}
\item If $\beta$ is hyperbolic negative
\begin{enumerate}[i.]
\item negative hyperbolic type one, $n$ is odd or $4k, k \in {\bf N^+}$
\item negative hyperbolic type two, $n$ is odd or $2$
\end{enumerate}
\end{enumerate}

\end{lemma}
\begin{proof}
Suppose a braid $\xi$ has multiplicity $n$ and winding number $\eta$: if gcd$(n, \eta) = 1$, then $\xi$ is isotopic to a $(n, \eta)$ torus braid, and this braid has writhe $\eta(n-1)$; if gcd$(n, \eta) = d > 1$, then $\xi$ is a $d$-strand cabling of a braid $\xi_1$ with multiplicity $\frac{n}{d}$ and winding number $\frac{\eta}{d}$, and we can get $\omega(\xi) > \eta(n-1)$.See \cite[Lemma 6.7]{Hutchings}.

If $\beta$ is hyperbolic positive type one, we can choose a trivialization such that $\mu_1(\beta^k) = \frac{1}{2}$, then $\xi$  has minimal winding number $1$, gcd$(n, 1) = 1$.

If $\beta$ is hyperbolic positive type two, we can choose a trivialization such that $\mu_1(\beta^k) = -\frac{1}{2}$, then $\xi$  has minimal winding number $0$, gcd$(n, 0) = n$.

If $\beta$ is hyperbolic negative type one, we can choose a trivialization such that
\begin{align*}
\mu_1(\beta^k) = \begin{cases}
\frac{k}{2},\quad  k\ \text{is odd}\\
\frac{k}{2} + \frac{1}{2},\quad  k\ \text{is even}
\end{cases}
\end{align*}
$\xi$  has minimal winding number $\eta = \mu_1(\beta^n) + \frac{1}{2}$, hence
\begin{align*}
\text{gcd}(n, \eta) = \begin{cases}
\text{gcd}(n, \frac{n}{2} + \frac{1}{2}) = 1,\quad  n\ \text{is odd}\\
\text{gcd}(n, \frac{n}{2} + 1) = \text{gcd}(2, \frac{n}{2}-1),\quad  n\ \text{is even}
\end{cases}
\end{align*}

If $\beta$ is hyperbolic negative type two, we can choose a trivialization such that
\begin{align*}
\mu_1(\beta^k) = \begin{cases}
\frac{k}{2},\quad  k\ \text{is odd}\\
\frac{k}{2} - \frac{1}{2},\quad  k\ \text{is even}
\end{cases}
\end{align*}
$\xi$ has minimal winding number $\eta = \mu_1(\beta^n) + \frac{1}{2}$, hence
\begin{align*}
\text{gcd}(n, \eta) = \begin{cases}
\text{gcd}(n, \frac{n}{2} + \frac{1}{2}) = 1,\quad  n\ \text{is odd}\\
\text{gcd}(n, \frac{n}{2}) = \frac{n}{2},\quad  n\ \text{is even}
\end{cases}
\end{align*}

By the argument in first paragraph, the equality holds only if $\text{gcd}(n, \eta) = 1$, we finish the proof.
\end{proof}

The linking number $\ell(\xi_1, \xi_2)$ is defined to be the signed number of crossings of $\xi_1$ with $\xi_2$, the same sign convention as for the writhe, See \cite[\S 6.3]{Hutchings}. By same proof as Hutchings, we can prove
\begin{lemma}
	Linking number $\ell(\xi_1, \xi_2) \geq \text{min} (q_1 (\mu_1(\xi_2) + \frac{1}{2}), q_2 (\mu_1(\xi_1) + \frac{1}{2}))$
\end{lemma}

Now we can give the proof of theorem \ref{Par}. Using the same reasoning as in \cite[\S 6.4, \S 6.5]{Hutchings}, we first reduce the theorem to the following lemma

\begin{lemma}
	Let $\beta$ be a brake orbit, $\{q_1, \ldots, q_k\}$ is a partition of $n$, i.e. $q_i \in {\bf N}^+, 1 < i < k, q_1 + \cdots + q_k = n$. Let $\xi$ be the braid corresponding to the partition at the negative ends. Then
	\begin{equation}
		\frac{1}{2}\omega(\xi) + \sum_{i = 1}^{k}\mu_1(\beta^{q_i}) \geq \sum_{i =1}^{n}\mu_1(\beta^i)
	\end{equation}
	Equality holds only if $\{q_1, \ldots, q_k\}$ is a partition as in theorem \ref{Par}.
\end{lemma}

By the same reasoning in \cite[\S 6.4, \S 6.5]{Hutchings} as well, we have
 \begin{equation}\label{f}
 \frac{1}{2}\omega(\xi) + \sum_{i = 1}^{k}\mu_1(\beta^{q_i}) \geq \sum_{i = 1}^{k}(\mu_1(\beta^{q_i})- \frac{1}{2}\rho_i) + \frac{1}{2}\sum_{i, j}^{k}\min(q_i \rho_j, q_j \rho_i)
 \end{equation}
This lemma is equivalent to the following form.

\begin{lemma}
		Let $\beta$ be a brake orbit, $\{q_1, \ldots, q_k\}$ is a partition of $n$, i.e. $q_i \in {\bf N}^+, 1 < i < k, q_1 + \cdots + q_k = n$. Let $\rho_i = \mu_1(\beta^{q_i}) + \frac{1}{2}$. Then
		\begin{equation}\label{RECH}
		\sum_{i = 1}^{k}(\mu_1(\beta^{q_i})- \frac{1}{2}\rho_i) + \frac{1}{2}\sum_{i, j}^{k}\min(q_i \rho_j, q_j \rho_i) \geq \sum_{i = 1}^{n}\mu_1(\beta^i)
		\end{equation}
		Equality holds only if $\{q_1, \ldots, q_k\}$ is a partition as in theorem \ref{Par}.
\end{lemma}

All we need to prove is the above lemma, because numbers in the partition of theorem \ref{Par} in each case belong to the numbers in lemma\ref{wind}, the inequality (\ref{f}) can take equality.

\begin{proof}
Firstly we know the following inequality from \cite[equation (49)]{Hutchings} for a nonsymmetric pseudoholomorphic curve. Let $\rho'_i = \lceil \frac{\mu_{CZ}(\beta^{q_i})}{2} \rceil$
\begin{equation}
\sum_{i = 1}^{k}(\mu_{CZ}(\beta^{q_i})- \rho'_i) + \sum_{i, j}^{k}\min(q_i \rho'_j, q_j \rho'_i) \geq \sum_{i = 1}^{n}\mu_{CZ}(\beta^i)
\end{equation}
Equality holds only if the partition satisfies the ECH partition.

We denote the left and right side of the above inequality by $\text{Left}_0$ and $\text{Right}_0$, and the left and right side of the inequality (\ref{RECH}) by Left and Right.

Secondly the validity of the inequality does not depend on the choice of the trivialization, because we take a different trivialization, both sides will add an integer multiple of $\frac{1}{2}(n^2 + n)$. Therefore we can choose the trivialization in hyperbolic case such that $\mu_1(\beta) = \frac{1}{2}$.

In the elliptic case, $\mu_1(\beta^k) = [k\theta] + \frac{1}{2}, \mu_{CZ} = 2[k\theta] + 1$. Therefore $\mu_1(\beta^{q_i}) = \frac{1}{2}\mu_{CZ}(\beta^{q_i}) = [q_i\theta] + \frac{1}{2}, \rho_i = \rho'_i = [q_i\theta] + 1$. The Left and Right are both half the counterpart number $\text{Left}_0, \text{Right}_0$ from the nonsymmetric case. Therefore we get the result from nonsymmetric case.
	
In the hyperbolic positive type two case, because of our choice of trivialisation $\mu_1(\beta) = \frac{1}{2}$, we have $\mu_{CZ}(\beta) = 2$ by relations $\mu_{CZ}(\beta) = \mu_1(\beta) + \mu_2(\beta), \mu_1(\beta) - \mu_2(\beta) = -1$. From our iteration formula for hyperbolic positive type two (\ref{hp2}), we have $\mu_1(\beta^k) = k\mu_1(\beta) + \frac{k-1}{2} = k - \frac{1}{2}$. At the same time we have $\mu_{CZ}(\beta^k) = k\mu_{CZ}(\beta) = 2k $. Hence we get $\mu_1(\beta^k) = \frac{1}{2}\mu_{CZ}(\beta^k) - \frac{1}{2}$. By definition $\rho_i = \mu_1(\beta^{q_i}) + \frac{1}{2} = q_i, \rho' _i= \lceil \frac{\mu_{CZ}(\beta^{q_i})}{2} \rceil =q_i$, we get $\rho'_i = \rho_i$. Using the relations $\mu_1(\beta^k) = \frac{1}{2}\mu_{CZ}(\beta^k) - \frac{1}{2}$ and $\rho'_i = \rho_i$, we get $\text{Left} = \frac{1}{2}\text{Left}_0 - \frac{k}{2}$ and $\text{Right} = \frac{1}{2}\text{Right}_0 - \frac{n}{2}$ by definition. Because $\text{Left}_0 \geq \text{Right}_0$ and $-\frac{k}{2} \geq -\frac{n}{2}$, equality holds only if the original partition holds.
	
In the hyperbolic positive type one case, because of our choice of trivialisation $\mu_1(\beta) = \frac{1}{2}$. From our iteration formula for hyperbolic positive type two (\ref{hp1}), we have $\mu_1(\beta^k) = k\mu_1(\beta) + \frac{1 - k}{2} = \frac{1}{2}$. Then Left $= \frac{1}{2}\sum_{i, j =1}^{k}\min(q_i, q_j)$ by definition. Without loss of generality, we can choose $q_1 \leq \ldots \leq q_k$, and Left $=\frac{1}{2} ((2k -1) q_1 + (2k - 3)q_2 + \cdots + q_k)$. At the same time, Right $= \frac{n}{2}$ by definition. Because $q_1 + \cdots q_k = n$, Left $\geq \frac{n}{2}$. The equality holds and only if $k = 1$.

In the hyperbolic negative type two case, because of our choice of trivialisation $\mu_1(\beta) = \frac{1}{2}$. By (\ref{hn2}), we have
\begin{equation*}
  \mu_1(\beta^r) = \left\{
  \begin{array}{ll}
    \frac{r}{2}, & \hbox{$r$ is odd;} \\
    \frac{r}{2} - \frac{1}{2}, & \hbox{$r$ is even.}
  \end{array}
\right.
\end{equation*}
 At the same time we have $\mu_{CZ}(\beta^k) = k\mu_{CZ}(\beta) = k $.
 By definition we have
 \begin{equation*}
  \rho_i = \mu_1(\beta^{q_i}) + \frac{1}{2} = \left\{
  \begin{array}{ll}
    \frac{q_i}{2} +\frac{1}{2}, & \hbox{$q_i$ is odd;} \\
    \frac{q_i}{2}, & \hbox{$q_i$ is even.}
  \end{array}
\right.
\end{equation*}
So we get $\rho_i = \lceil \frac{\mu_{CZ}(\beta^{q_i})}{2} \rceil$. Then by definition $\rho' _i= \lceil \frac{\mu_{CZ}(\beta^{q_i})}{2} \rceil $ which is equal to $\rho_i$. We have the relations: $\mu_1(\beta^k) = \frac{1}{2}\mu_{CZ}(\beta^k)$, if $k$ is odd, $\mu_1(\beta^k) = \frac{1}{2}\mu_{CZ}(\beta^k) - \frac{1}{2}$, if $k$ is even and $\rho'_i = \rho_i$. We get $\text{Left} = \frac{1}{2}\text{left}_0 - \frac{1}{2}\#\text{even number}$, where $\#\text{even number}$ means the cardinality of even number in the partition number, and $\text{Right} = \frac{1}{2}\text{Right}_0 - \frac{1}{2} [\frac{n}{2}]$ by definition. Because $\text{Left}_0 \geq \text{Right}_0$ and $-\frac{1}{2}\#\text{even number} \geq -  [\frac{n}{2}]$. The equality holds only if the original partition holds.
	
In the hyperbolic negative type one case, because of our choice of trivialisation $\mu_1(\beta) = \frac{1}{2}$. By (\ref{hn1}) we have
\begin{equation*}
  \mu_1(\beta^r) = \left\{
  \begin{array}{ll}
    \frac{r}{2}, & \hbox{$r$ is odd;} \\
    \frac{r}{2} + \frac{1}{2}, & \hbox{$r$ is even.}
  \end{array}
\right.
\end{equation*}
By definition we have
\begin{equation*}
  \rho_i = \mu_1(\beta^{q_i}) + \frac{1}{2} =\left\{
  \begin{array}{ll}
    \frac{q_i}{2} + \frac{1}{2}, & \hbox{$r$ is odd;} \\
    \frac{q_i}{2} + 1, & \hbox{$r$ is even.}
  \end{array}
\right.
\end{equation*}
$\mu_1(\beta^{q_i}) - \frac{1}{2}\rho_i = \frac{q_i}{4} - \frac{1}{4}$, if $q_i$ is odd; $\mu_1(\beta^{q_i}) - \frac{1}{2}\rho_i = \frac{q_i}{4}$, if $q_i$ is even. So Left $= \frac{1}{4}(q_1 + \cdots + q_k) - \frac{1}{4}(\# \text{odd number}) + \frac{1}{2}\sum_{i, j = 1}^{k}\min(q_i\rho_j, q_j\rho_i)$ by definition, where $\#\text{odd number}$ means the cardinality of odd number in the partition number. Since $\rho_i \geq \frac{q_i}{2} + \frac{1}{2}$, we have $\frac{1}{2}\min(q_i\rho_j, q_j\rho_i) \geq \frac{1}{2}\min(q_i(\frac{q_j}{2} + \frac{1}{2}), q_j(\frac{q_i}{2} + \frac{1}{2}))$, Left $\geq \frac{n}{4} - \frac{1}{4}(\# \text{odd number}) + \frac{1}{4}\sum_{i, j = 1}^{k}(q_i q_j) + \frac{1}{2}\sum_{i, j = 1}^{k}\min(\frac{q_i}{2}, \frac{q_j}{2})$. Without loss of generality, we can choose $q_1 \leq \ldots \leq q_k$, then Left $\geq \frac{n}{4} - \frac{1}{4}(\# \text{odd number}) + \frac{n^2}{4} + \frac{1}{4}((2k-1)q_1 + (2k-3)q_2 \cdots + q_k)$. By relations $q_1+\cdots + q_k = n$, $q_i \geq 1$ and $\# \text{odd number} \leq k$, we get $\text{Left} \geq \frac{n}{4} - \frac{1}{4}k + \frac{n^2}{4} + \frac{1}{4}n + \frac{1}{4}(k^2 - k) = \frac{n(n+1)}{4} + \frac{n}{4} + \frac{k(k-2)}{4}$.
At the same time by definition
\begin{equation*}
  \text{Right} = \frac{n(n+1)}{4} + \frac{1}{2}[\frac{n}{2}] = \left\{
  \begin{array}{ll}
    \frac{n(n+1)}{4} + \frac{n}{4} - \frac{1}{4}, & \hbox{$n$ is odd;} \\
   \frac{n(n+1)}{4} + \frac{n}{4}, & \hbox{$n$ is even.}
  \end{array}
\right.
\end{equation*}

If $k > 2$, we get Left $>$ Right. And if $k=2$ and $n$ is odd, we get Left $>$ Right.

When $k=2$, $n$ is even and $\# \text{odd number} = 0$,  Left $\geq \frac{n}{4} - \frac{1}{4}(\# \text{odd number}) + \frac{n^2}{4} + \frac{1}{4}(3q_1 + q_2) = \frac{n(n+1)}{4} + \frac{n}{4} + \frac{q_1}{2} >$ Right.

When $k=2$, $n$ is even and $\# \text{odd number} = 2$, Left $= \frac{n}{4} - \frac{1}{4}(\# \text{odd number}) + \frac{n^2}{4} + \frac{1}{4}(3q_1 + q_2) = \frac{n(n+1)}{4} + \frac{n}{4} + \frac{q_1 - 1}{2} \geq$ Right, the equality holds only if $q_1 = 1$.

When $k = 1$,
\begin{equation*}
\text{Left} = \mu_1(\beta^n) - \frac{1}{2}\rho_1 + \frac{1}{2}n\rho_1 = \left\{
\begin{array}{ll}
\frac{n(n+1)}{4} + \frac{n}{4} - \frac{1}{4}, & \hbox{$n$ is odd;} \\
\frac{n(n+1)}{4} + \frac{n}{4} + \frac{n}{4}, & \hbox{$n$ is even.}
\end{array}
\right.
\end{equation*}
Only if $n$ is odd, Left $=$ Right.

\end{proof}

\section{Index inequality of Real multiple covers}
Let $u$ be a Real pseudoholomorphic curve of genus $0$, which is a Real branched cover of a somewhere injective Real pseudoholomorphic curve $\bar{u}$. Let $D$ denote the covering multiplicity, and $B$ the total branch number of this cover. Suppose $u$ has $k$ positive symmetric punctures, $l$ negative symmetric punctures, $m$ pairs of nonsymmetric positive punctures, and $n$ pairs of nonsymmetric negative punctures. Similarly suppose $\bar{u}$ has $\bar{k}$ positive symmetric punctures, $\bar{l}$ negative symmetric punctures, $\bar{m}$ pairs of nonsymmetric positive punctures, and $\bar{n}$ pairs of nonsymmetric negative punctures.

Suppose that $\bar{u}$ has symmetric positive ends at brake orbits $\alpha_1, \ldots, \alpha_{\bar{k}}$, symmetric negative ends at brake orbits $\beta_1, \ldots, \beta_{\bar{l}}$, pairs of symmetric positive ends $\gamma_1, \ldots, \gamma_{\bar{m}}$, and pairs of symmetric negative ends $\delta_1, \ldots, \delta_{\bar{n}}$. Then by \cite[Main Theorem]{ZZ}
\begin{equation*}
\text{ind}_R(\bar{u}) = -\frac{1}{2}\chi(\bar{u}) + \sum_{i=1}^{\bar{k}}\mu_1(\alpha_i) - \sum_{j=1}^{\bar{l}}\mu_1(\beta_j) + \sum_{p=1}^{\bar{m}}\mu_{CZ}(\gamma_{q}) - \sum_{q=1}^{\bar{n}}\mu_{CZ}(\delta_{q})
\end{equation*}
We can get the $\text{ind}_R(u)$ in the same way
\begin{equation*}
\text{ind}_R(u) = -\frac{1}{2}\chi(u) + \sum_{i=1}^{k}\mu_1(\zeta_i) - \sum_{j=1}^{l}\mu_1(\eta_j) + \sum_{p=1}^{m}\mu_{CZ}(\xi_{q}) - \sum_{q=1}^{n}\mu_{CZ}(\chi_{q})
\end{equation*}
Where the symmetric positive end of $\bar{u}$, $\alpha_*$ is covered by symmetric positive ends of $u$, $\zeta_*$ or pair of nonsymmetric positive ends of $u$, $\xi_*$ with total multiplicity $D$, each negative end of $\bar{u}$, $\beta$ is covered by  symmetric negative ends of $u$, $\eta_*$ or pair of nonsymmetric negative ends of $u$, $\chi_*$ in total multiplicity $D$ as well. A pair of positive nonsymmetric ends $\gamma_*$ can only be covered by pair of nonsymmetric positive ends $\xi_*$ with total multiplicity $D$, a pair of nonsymmetric negative ends $\delta_*$ can only be covered by pair of nonsymmetric negative ends $\chi_*$ with total multiplicity $D$.

Thanks to our index iteration formulae,  we can get the following inequalities.

\begin{lemma}\label{le3}
Let $\alpha$ be a brake orbit, then
\begin{equation}\label{1}
k\mu_1(\alpha) - \frac{k-1}{2} \leq \mu_1(\alpha^k) \leq k\mu_1(\alpha) + \frac{k-1}{2}
\end{equation}
\end{lemma}
\begin{proof}
	If $\alpha$ is hyperbolic, from equations (\ref{hn1})( \ref{hn2})( \ref{hp1})( \ref{hp2}) we know
	\begin{equation*}
	k\mu_1(\alpha) - \frac{k-1}{2} \leq \mu_1(\alpha^k) \leq k\mu_1(\alpha) + \frac{k-1}{2}
	\end{equation*}
	If $\alpha$ is elliptic, $\alpha = R(2\pi\theta')$, then
	\begin{align*}
	\mu_1(\alpha^k) = [k\theta] + \frac{1}{2} \\
	k\mu_1(\alpha)= k[\theta] + \frac{k}{2}
	\end{align*}
	Because $0 \leq [k\theta] - k[\theta] \leq k-1$, we have $|\mu_1(\alpha^k) - k\mu_1(\alpha)| \leq \frac{k -1}{2}$
\end{proof}

Using iteration formulae for Conley-Zehnder index and relations between $\mu_1$ and $\mu_{CZ}$ we obtain
\begin{lemma}\label{le4}
\begin{align}
\begin{cases}
\mu_{CZ}(\alpha^{k}) = 2k\mu_1(\alpha), \alpha\ \text{is negative hyperbolic}\\
\mu_{CZ}(\alpha^{k}) = 2k\mu_1(\alpha) - k, \alpha\ \text{is positive hyperbolic type one}\\
\mu_{CZ}(\alpha^{k}) = 2k\mu_1(\alpha) + k, \alpha\ \text{is positive hyperbolic type two}\\
1-k \leq \mu_{CZ}(\alpha^{k}) - 2k\mu_1(\alpha) \leq k-1, \alpha\ \text{is elliptic}
\end{cases}
\end{align}
\end{lemma}

\begin{proof}
	If $\alpha$ is hyperpolic, from iteration formulae it is obvious.
	
	If $\alpha$ is elliptic, $\alpha = R(2\pi\theta')$, then $\mu_{CZ}(\alpha^{k}) = 2[k\theta] + 1, \mu_1(\alpha) = [\theta] + \frac{1}{2}$. From the inequality $0 \leq 2[k\theta] - 2k[\theta] \leq 2k-2$, we get the result.
	
\end{proof}

Let $\#_1$ denote the number of positive pairs of ends of $u$, $\xi_* = \alpha^k$, where $\alpha$ is a brake orbit of hyperbolic positive type one,  $\xi_*$ covers a positive symmetric end $\alpha$ at positive ends; Let $\#_2$ denote the number negative pairs of ends of $u$, $\chi_* = \beta^k$, $\beta$ is a brake orbit of hyperbolic positive type two,  $\chi_*$ covers a positive symmetric end $\beta$ at negative ends. For such $\xi_*$ and $\chi_*$, we have by equation (\ref{le4})
\begin{align}
\mu_{CZ}(\xi_*) = \mu_{CZ}(\alpha^k) = 2k\mu_1(\alpha) - (k-1) -1, \label{3}\\
\mu_{CZ}(\chi_*) = \mu_{CZ}(\beta^k) = 2k\mu_1(\beta) + (k-1) +1 \label{4}
\end{align}

By the iteration formula in \cite[\S 10.1 equation (19)]{Long}, we have
\begin{equation}\label{5}
k\mu_{CZ}(\alpha) - (k-1) \leq \mu_{CZ}(\alpha^{k}) \leq k\mu_{CZ}(\alpha) + (k-1)
\end{equation}

By Riemann-Hurwitz we have
\begin{equation*}
\chi(u) = D\chi(\bar{u}) - B
\end{equation*}
which means that
\begin{equation*}
2-k-l-2m-2n = D(2 - \bar{k} - \bar{l} - 2\bar{m} - 2\bar{n}) - B
\end{equation*}

Suppose a brake orbit $\beta$ is covered by $s$ symmetric orbits, $(\beta^{p_1}, \ldots, \beta^{p_s})$ and $t$ pairs of nonsymmetric orbits, $(\beta^{q_1}, \ldots, \beta^{q_t})$ with total covering multiplicity $D = \sum_{i=1}^{s}p_i + \sum_{j = 1}^{t}2q_j$, then if $\beta$ is not hyperbolic positive type two, by equations (\ref{le3})( \ref{le4})
\begin{align*}
&\mu_1(\beta^{p_1}) + \ldots + \mu_1(\beta^{p_s}) + \mu_{CZ}(\beta^{q_1}) + \ldots + \mu_{CZ}(\beta^{q_t}) \\
&\leq \sum_{i = 1}^{s}(p_i\mu_1(\beta) + \frac{p_i - 1}{2}) + \sum_{j = 1}^{t}(2q_j\mu_1(\beta) + q_j - 1)\\
&= D\mu_1(\beta) + \frac{D}{2} - \frac{s}{2} -t
\end{align*}
and if $\beta$ is not hyperbolic positive type one, by equations (\ref{le3})( \ref{le4})
\begin{align*}
&\mu_1(\beta^{p_1}) + \ldots + \mu_1(\beta^{p_s}) + \mu_{CZ}(\beta^{q_1}) + \ldots + \mu_{CZ}(\beta^{q_t}) \\
&\geq \sum_{i = 1}^{s}(p_i\mu_1(\beta) - \frac{p_i - 1}{2}) + \sum_{j = 1}^{t}(2q_j\mu_1(\beta) - q_j + 1)\\
&= D\mu_1(\beta) - \frac{D}{2} + \frac{s}{2} + t
\end{align*}

if $\beta$ is hyperbolic positive type two, by equations (\ref{le3})( \ref{le4})
\begin{align*}
&\mu_1(\beta^{p_1}) + \ldots + \mu_1(\beta^{p_s}) + \mu_{CZ}(\beta^{q_1}) + \ldots + \mu_{CZ}(\beta^{q_t}) \\
&\leq \sum_{i = 1}^{s}(p_i\mu_1(\beta) + \frac{p_i - 1}{2}) + \sum_{j = 1}^{t}(2q_j\mu_1(\beta) + q_j)\\
&= D\mu_1(\beta) + \frac{D}{2} - \frac{s}{2}
\end{align*}

if $\beta$ is hyperbolic positive type one, by equations (\ref{le3})( \ref{le4})
\begin{align*}
&\mu_1(\beta^{p_1}) + \ldots + \mu_1(\beta^{p_s}) + \mu_{CZ}(\beta^{q_1}) + \ldots + \mu_{CZ}(\beta^{q_t}) \\
&\geq \sum_{i = 1}^{s}(p_i\mu_1(\beta) - \frac{p_i - 1}{2}) + \sum_{j = 1}^{t}(2q_j\mu_1(\beta) - q_j)\\
&= D\mu_1(\beta) - \frac{D}{2} + \frac{s}{2}
\end{align*}

Let $\alpha$ be a pair of nonsymmetric periodic Reeb orbit, which is covered by $t$ pairs of nonsymmetric periodic orbits $(\alpha^{q_1}, \ldots, \alpha^{q_t})$ with total multiplicity $D = q_1 + \dots + q_t$, by the same reason as above, we have
\begin{equation*}
D\mu_{CZ}(\alpha) - (D-1) \leq \mu_{CZ}(\alpha^{q_1}) + \ldots + \mu_{CZ}(\alpha^{q_t}) \leq D\mu_{CZ}(\alpha) + (D-1)
\end{equation*}

Because each end of $\bar{u}$ is covered $D$ times, adding the above estimates at each puncture of $\bar{u}$ together, we get
\begin{equation}
\begin{split}
\text{ind}_R(u) & = -\frac{1}{2}\chi(u) + \sum_{i=1}^{k}\mu_1(\zeta_i) - \sum_{j=1}^{l}\mu_1(\eta_j) + \sum_{p=1}^{m}\mu_{CZ}(\xi_{q}) - \sum_{q=1}^{n}\mu_{CZ}(\chi_{q})\\
& \geq -\frac{1}{2}(D\chi(\bar{u}) - B) + (D\sum_{i=1}^{\bar{k}}\mu_1(\alpha_i) + D\sum_{p=1}^{\bar{m}}\mu_{CZ}(\gamma_{p}) - \frac{D\bar{k} - k}{2} - (D\bar{m} - m)- \#_1)\\
& - (D\sum_{j=1}^{\bar{l}}\mu_1(\beta_j) + D\sum_{q=1}^{\bar{n}}\mu_{CZ}(\delta_{q}) + \frac{D\bar{l} - l}{2} + (D\bar{n} - n) + \#_2)\\
& = -\frac{D}{2}(\chi(\bar{u}) + \sum_{i=1}^{\bar{k}}\mu_1(\alpha_i) - \sum_{j=1}^{\bar{l}}\mu_1(\beta_j) + \sum_{p=1}^{\bar{m}}\mu_{CZ}(\gamma_{q}) - \sum_{q=1}^{\bar{n}}\mu_{CZ}(\delta_{q})) + \frac{B}{2}\\
& + (- \frac{D\bar{k} - k}{2} - (D\bar{m} - m) - \frac{D\bar{l} - l}{2} - (D\bar{n} - n))-\#_1-\#_2\\
& = D\text{ind}_R(\bar{u})+ \frac{B}{2} + \frac{(B + 2 -2D)}{2} -\#_1-\#_2 \ (\text{By Riemann-Hurwitz})\\
& = D\text{ind}_R(\bar{u}) + (B+1-D) -\#_1-\#_2
\end{split}
\end{equation}

So we get the theorem, which is similar to \cite[Lemma 2.2]{HN}.
\begin{theo}\label{cover}
	Let $u$ be a Real pseudoholomorphic curve with genus $0$, which covers a somewhere injective Real pseudoholomorphic curve $\bar{u}$. Let $D$ denote the covering multiplicity and $B$ the total branch number of $u$ over $u'$. Then
	\begin{equation}
	\text{ind}_R(u) \geq D\text{ind}_R(\bar{u}) + (B+1-D) - \#_1 - \#_2
	\end{equation}
Where $\#_1$ is the number of hyperbolic positive type one pairs which cover a brake orbit at positive asymptotics; $\#_2$ is the number of hyperbolic positive type two pairs which cover a brake orbit at negative asymptotics.
\end{theo}

\appendix
\section{application}\label{Appendix 1}
In this appendix we give possible applications of the above inequalities to the construction of Real cylindrical contact homology.

\begin{lemma}\label{lem1}
	Let $u$ be a genus zero Real pseudoholomorphic curve with one positive symmetric puncture, $l$ negative symmetric punctures and $n$ pairs of nonsymmetric negative punctures. Suppose that the somewhere injective Real curve $\bar{u}$ underlying $u$ is a nontrivial cylinder. Then
	\begin{equation*}
	\text{ind}_R u \geq l + 2n -\#_2
	\end{equation*}
\end{lemma}
\begin{proof}
	By Riemman-Hurwitz, $1-l-2n = -B$. Since $\bar{u}$ is nontrivial, $ind_R(\bar{u}) \geq 1$. Then by equation (\ref{cover})
	\begin{equation*}
	\text{ind}_R(u) \geq D + (B+1-D) -\#_2 = 1+ B -\#_2 \geq l + 2n -\#_2
	\end{equation*}
\end{proof}

\begin{lemma}\label{lem2}
	Let $u$ be a genus zero Real pseudoholomorphic curve with one positive symmetric puncture, $l > 1$ negative symmetric punctures and $n$ pairs of nonsymmetric negative punctures. Suppose that $u$ is not a multiple cover of a cylinder. Then
	\begin{equation*}
	\text{ind}_R(u) \geq 1 -\#_2
	\end{equation*}
\end{lemma}
\begin{proof}

	By theorem \ref{cover} $\text{ind}_R(u) \geq D\text{ind}_R(\bar{u}) + (B+1-D) - \#_2$ and $ind_R(\bar{u}) \geq 1$,
	\begin{equation*}
	\text{ind}_R(u) \geq D + (B+1-D) -\#_2 = 1+ B -\#_2 \geq 1 -\#_2
	\end{equation*}

\end{proof}

\begin{lemma}\label{lem3}
	Let $u$ be a Real pseudoholomorphic cylinder which covers $\bar{u}$, a somewhere injective cylinder. The covering multiplicity is $D$. Then
	\begin{equation*}
	1 \leq \text{ind}_R(\bar{u}) \leq \text{ind}_R(u)
	\end{equation*}
\end{lemma}
\begin{proof}
	Let $\alpha$ and $\beta$ be the positive and negative brake orbits. Choose a trivialization so that $c_1(\bar{u}) = 0$, then
	\begin{equation*}
	\text{ind}_R(\bar{u}) = \mu_1(\alpha) - \mu_1(\beta)
	\end{equation*}
	and
	\begin{equation*}
	\text{ind}_R(u) = \mu_1(\alpha^D) - \mu_1(\beta^D)
	\end{equation*}
	Since $\bar{u}$ is not a trivial cylinder, ind$\bar{u} \geq 1$. By equation (\ref{1}), we have
	\begin{equation*}
	\mu_1(\alpha^D) \geq D\mu_1(\alpha) - \frac{D -1}{2},\quad \mu_1(\beta) \leq D\mu_1(\beta^D) + \frac{D -1}{2}
	\end{equation*}
	So we get
\begin{align*}
  \text{ind}_R(u) &\geq D(\mu_1(\alpha) - \mu_1(\beta)) - (D - 1) \\
   & = \mu_1(\alpha) - \mu_1(\beta) + (D - 1)(\mu_1(\alpha) - \mu_1(\beta) - 1)\\
   & \geq \mu_1(\alpha) - \mu_1(\beta) = \text{ind}_R(\bar{u})
\end{align*}
\end{proof}

When we consider the compactness of a series of Real pseudoholomorphic curves, the limit will be a Real pseudoholomorphic building $u = (u_1, \ldots, u_k)$. Each component is a Real pseudoholomorphic curve $Nu_i(\cdot) = u_i(N\cdot)$, not necessary connected. The negative ends of $u_i$ are the same as the positive ends of $u_{i + 1}$. And the genus of $u$ is the genus of the Real pseudoholomorphic curve which is obtained by gluing together all components. We define the Fredholm index for a Real pseudoholomorphic building to be $ind_R(u) = \sum_{i = 1}^{k}ind_R(u_i)$. See reference \cite{Compact}.

A brake orbit $\beta$ on a convex hypersurface in ${\bf R}^4$ has the property $\mu_1(\beta) \geq \frac{3}{2}$, see \cite[Proposition 3.8]{Urs}. We call a contact form with anticontact involution dynamically convex, if any brake orbit $\beta$ on it satisfies the condition $\mu_1(\beta) \geq \frac{3}{2}$ and any periodic Reeb orbit $\gamma$ satisfies the condition $\mu_{CZ}(\gamma) \geq 3$.

\begin{theo}
Assume the contact form with anticontact involution $(\lambda, N)$ is dynamically convex and the almost complex structure $J$, which satisfies $JN = -NJ$, is generic. Suppose the Real pseudoholomorphic building $u = (u_1, \ldots, u_k)$ is a genus $0$ building with one symmetric positive puncture and no negative puncture. Then $ind_R(u) \geq 1$, and if $ind_R(u) = 1$, then $k=1$ and $u$ is a plane.
\end{theo}
\begin{proof}
If $k = 1$, then $u$ is a plane and by \cite[Main Theorem]{ZZ} $ind_R(u) = -\frac{1}{2} + \mu_1(\beta) \geq 1$ by the dynamically convex condition.

If $k > 1$, suppose the theorem is true for $k - 1$.

If $u_1$ has only one negative puncture, then $u_1$ must be a nontrivial cylinder, by induction we have that the building $u' = (u_2, \ldots, u_k)$ has index $ind_R(u') \geq 1$. So we get $ind_R(u) \geq ind_R(u_1) + 1 \geq 2$. The theorem holds.

 Suppose $u_1$ has at least two negative punctures. We assume $u_1$ has one positive, $l$ negative symmetric punctures and $n$ pairs of negative punctures . We have $l \geq 2$ or $n \geq 1$.

In \cite[Proposition 2.7]{HN}, the authors proved under the dynamically convex condition that a pseudoholomorphic building with one positive and no negative puncture has index bigger or equals to $2$, and the equality holds only if the building has $1$ level. In our case, the $n$ pairs of nonsymmetric negative punctures in $u_1$ will contribute $2n$ to the index. Hence by induction we have $ind_R(u) \geq ind_R(u_1) + l + 2n$.

Let $\bar{u}_1$ be the somewhere injective Real pseudoholomorphic curve underlying $u_1$. If $\bar{u}_1$ is a trivial cylinder, then by equation (\ref{RECHLE}) $ind_R(u_1) \geq 0$, and therefore $ind_R(u) \geq l +2n \geq 2$; If $\bar{u}_1$ is a nontrivial cylinder, then by equation (\ref{lem1}) $ind_R(u_1) \geq l + 2n - \#_2 \geq l + n$, because of $n  \geq \#_2$, and $ind_R(u) \geq l + n + l +2n \geq 2$; If $\bar{u}_1$ is not a trivial cylinder, then by equation (\ref{lem2}) $ind_R(u_1) \geq 1 - \#_2$, and because of $n  \geq \#_2$, $ind_R(u) \geq 1 - \#_2 + l +2n \geq 1+ l + n \geq 2$.
\end{proof}

We can rule out some bad breakings for the construction of Real cylindrical contact homology as well.

Assume $(\lambda, N)$ is dynamically convex , let $\{u_n\}$ be a sequence of Real pseudoholomorphic cylinders with Fredholm index for the Real pseudoholomorphic curve $1$, then the only nontrivial limit of $\{u_n\}$ will be the building $u = (u_1, u_2)$, where $u_1$ is a Real pseudoholmorphic curve with one positive symmetric puncture and two negative symmetric punctures, and $ind_R(u_1) = 0$, which is a $d+1$-multiple cover of a trivial cylinder ${\bf R} \times \beta$, $\beta$ is a brake orbit, the positive puncture of $u_1$ converges to $\beta^{d+1}$, one negative puncture converges to $\beta^d$ and the other negative puncture $\beta$; $u_2$ consist of a trivial cylinder ${\bf R} \times \beta^d$ and a plane with one positive puncture, which converges to the brake orbit $\beta$. There exists no pair of nonsymmetric punctures, because of dynamically convex condition, $\mu_{CZ}(\gamma) \geq 3$.

\begin{theo}
Let the building $u = (u_1, u_2)$ be as follows: $u_1$ is a Real pseudoholmorphic curve with one positive symmetric puncture and two negative symmetric punctures, and $\textrm{ind}_\textrm{R}(u_1) = 0$, which is a $d+1$-multiple Real cover of a trivial Real cylinder ${\bf R} \times \beta$, $\beta$ is a brake orbit, the positive symmetric puncture of $u_1$ converges to $\beta^{d+1}$ one negative symmetric puncture converges to $\beta^d$ and the other negative symmetric puncture converges to $\beta$; $u_2$ consists of a trivial Real cylinder ${\bf R} \times \beta^d$ and a Real plane with one symmetric positive puncture, which converges to the brake orbit $\beta$, then a sequence of Real pseudoholomorphic cylinders $\{u_n\}$ in ${\mathcal M}^1_R(\beta^{d+1}, \beta^d)$ will not converge to $u= (u_1, u_2)$.
\end{theo}

\begin{proof}
Firstly because of $ind_R(u) = 1$, we have $\mu_1(\beta^{d+1}) - \mu_1(\beta^d) = 1$

Let $\zeta^+$ denote the braid corresponding to the positive end of $u_1$ at $\beta^{d + 1}$,  $\zeta^-$ denote the braid corresponding to the negative end of $u_1$, $\zeta$ consists of two components $\zeta = \zeta_1 \bigcup \zeta_2$, $\zeta_1$ at $\beta^d$ and $\zeta$ at $\beta$.

From our discussion in section 3, we have that the writhe satisfies the inequality $\omega(\zeta^+)$, $\omega(\zeta^+) \leq d(\mu_1(\beta^{d+1})-\frac{1}{2})$. Because we can choose $\zeta_1$ within distance $\varepsilon$ of $\beta$ and $\zeta_2$ has distance at least $2\varepsilon$ from $\beta$, then $\omega(\zeta^-) = \omega(\zeta_1) + 2d \textrm{wind}(\zeta_2) + \omega(\zeta_2)$. Note that $\omega(\zeta_2) = 0$, because $\xi_2$ has degree $1$. Moreover by lemma \ref{wind} $\textrm{wind}(\zeta_2) \geq \mu_1(\beta) + \frac{1}{2}$, and $\omega(\zeta_1) \geq (d-1)(\mu_1(\beta^d)+\frac{1}{2})$.

By the adjunction formula of \cite[Lemma 3.5]{HN}, we have $\chi(u_1) + \omega(\zeta^+) - \omega(\zeta^-) = 2\Delta(u_1) \geq 0$, where $\Delta(u_1)$ is the singularity number.

Since $u_1$ is a pair of pants, we have $\chi (u_1) = -1$. Inserting the estimate of the writhe into the adjunction formula, we get
\begin{equation*}
  -1 + d(\mu_1(\beta^{d+1})-\frac{1}{2}) - ((d-1)(\mu_1(\beta^d) + \frac{1}{2}) + 2d(\mu_1(\beta) + \frac{1}{2})) \geq 0
\end{equation*}
Using the relation $\mu_1(\beta^{d+1}) - \mu_1(\beta^d) = 1$, we have
\begin{align*}
0 \leq & -1 + d(\mu_1(\beta^{d+1})-\frac{1}{2}) - \left((d-1)(\mu_1(\beta^d) + \frac{1}{2}) + 2d(\mu_1(\beta) + \frac{1}{2})\right) \\
  = & -1 + \mu_1(\beta^{d+1})-\frac{1}{2} + (d-1)(\mu_1(\beta^{d+1}) - \mu_1(\beta^d) -1) - 2d(\mu_1(\beta) +\frac{1}{2})\\
  = & -1 + \mu_1(\beta^{d+1})-\frac{1}{2} -2d(\mu_1(\beta) + \frac{1}{2})
\end{align*}
Finally, we get $\mu_1(\beta^{d+1}) \geq d(2\mu_1(\beta) + 1) + \frac{3}{2}$.

We have estimate $(d+1)\mu_1(\beta) + \frac{d}{2} \geq \mu_1(\beta^{d+1})$ equation (\ref{1}) in \S 4
\begin{align*}
  (d+1)\mu_1(\beta) + \frac{d}{2} & \geq \mu_1(\beta^{d+1}) \geq  d(2\mu_1(\beta) + 1) + \frac{3}{2}\\
  -\frac{d}{2} - \frac{3}{2}& \geq (d-1)\mu_1(\beta) \geq 0\\
\end{align*}
For the last inequality we have used the assumption of dynamical convexity condition. Because $d \geq 1$, we have a contradiction.
\end{proof}

In order to construct Real cylindrical contact homology for dynamical convex contact forms, we have to rule out all break cases of a sequence of Real cylinders with index $2$ except for breaking into two index $1$ cylinders. Almost all bad cases contain a index $0$ Real pseudoholomorphic curve, which is Real multiple cover of a trivial Real cylinder. Those cases can be ruled out by using the adjunction formula, interested readers can verify it by themselves.

There is just one bad breaking which can not be ruled out by using the adjunction formula. Let $\{u_n\}$ be a sequence of Real pseudoholomorphic cylinders in ${\mathcal M}^2_R(\alpha, \beta)$, where $\alpha, \beta$ are brake orbits. The limit breaks into a building $u=(u_1, u_2)$, where $u_1$ is a index $1$ pair of pants with one positive symmetric puncture which converges to a brake orbit $\alpha$, two negative symmetric punctures, which converge to brake orbits $\beta, \gamma$. $\gamma$ is a brake orbit with $\mu_1(\gamma) = \frac{3}{2}$. $u_2$ consists of a trivial cylinder $v_1 = {\bf R} \times \beta$ and a plane $v_2$ with one positive puncture, which converges to $\gamma$. We still have the adjunction formula for $u_1$, but $\alpha, \beta, \gamma$ have no connection, therefore we can not get a contradiction. How to take care of this bad case, we leave as a task for future researches.

\section{GIT description}\label{Appendix 2}

Here we give the GIT description of all cases by $\Psi_\beta(\tau)$.

Let $\beta$ be a brake orbit, and $\Psi_\beta(t)$ the corresponding symplectic path. We denote

$$\Psi_\beta(\tau) = \begin{pmatrix}
a & b\\
c & d
\end{pmatrix},
\Psi_\beta(\frac{\tau}{2}) = \begin{pmatrix}
u & v\\
w & x
\end{pmatrix}$$

Because $\Psi_\beta(\tau) = N\Psi_\beta(\frac{\tau}{2})^{-1}N\Psi_{\beta}(\frac{\tau}{2})$ and $xu - vw = 1$
\begin{equation*}
\begin{pmatrix}
a & b\\
c & d
\end{pmatrix}=
\begin{pmatrix}
1 + 2vw & 2vx\\
2uw & 1 + 2vw
\end{pmatrix}
\end{equation*}
we know $a = d$.

The nondegeneracy condition means that trace$(\Psi_{\beta}(\tau)) \neq 2$ which is equivalent to $vw \neq 0 \Longleftrightarrow v \neq 0, w \neq 0$.

Remark: in \cite{LZ}, the authors defined the nullities $\nu_1(\beta) = \nu^{RS}(L_0, \Psi([0, \frac{\tau}{2}])L_0)$, $\nu_2(\beta) = \nu^{RS}(L_1, \Psi([0, \frac{\tau}{2}])L_1)$. And the nullity for the Conley-Zehnder Index is $\nu_{CZ}(\alpha) = \nu^{RS}(W, Gr(\Psi([0, \tau])))$, where $W = \{ (x, x) \in {{\bf{R}}^{4n}}| x \in {\bf R}^{2n} \}$. There is the relation $\nu_{CZ} = \nu_1 + \nu_2$. We can see from the matrix $\Psi_{\beta}(\frac{\tau}{2})$, $\nu_1 = 0 \Longleftrightarrow v \neq 0$ and $\nu_2 = 0 \Longleftrightarrow w \neq 0$.

The elliptic case is characterized by the condition $-1 < \text{trace}(\Psi_{\beta}(\tau)) < 1$ which is equivalent to $-1 < a < 1$ or equivalently $-1 < vw < 0$.

$\beta$ is hyperbolic negative, if $a < -1$ or equivalently $vw < -1$; $\beta$ is hyperbolic positive, if $a > 1$ or equivalently $vw > 0$.
From the definition $\mu_1(\beta^2) = \mu^{RS}(L_0, \Psi([0,\tau])L_0)$, $\mu_2(\beta^2) = \mu^{RS}(L_1, \Psi([0, \tau])L_1)$, we infer that the four hyperbolic cases are determined by the signs of $a, b,c$.

Let $\beta$ be hyperbolic positive, then $a >1, bc=a^2-1 > 0$. If $b, c < 0$, then $\beta$ is hyperbolic positive type one $\mu_1(\beta) - \mu_2(\beta) =  1$. Equivalently we can determine the condition using the signs of $u, v, w, x$. $a > 1, b < 0 \Longleftrightarrow vw> 0, vx>0$. And we have $ux= vw+ 1>0$. So we get $u > 0, v < 0, w < 0, x > 0$ or $u < 0, v > 0, w > 0, x < 0$. If $b, c > 0$, then $\beta$ is hyperbolic positive type two $\mu_1(\beta) - \mu_2(\beta) =  -1$. Equivalently we can determine the condition using the signs of $u, v, w, x$. $a > 1, b > 0 \Longleftrightarrow vw> 0, vx>0$. And we have $ux= vw+ 1>0$. So we get $u > 0, v > 0, w > 0, x > 0$ or $u < 0, v < 0, w < 0, x < 0$. There are two choice of the signs of $u, v, w, x$, because of $\Psi_\beta(\tau) = N(-\Psi_\beta(\frac{\tau}{2}))^{-1}N(-\Psi_{\beta}(\frac{\tau}{2}))$, $\Psi_\beta(\frac{\tau}{2})$ and $-\Psi_\beta(\frac{\tau}{2})$ determine the same $\Psi_\beta(\tau)$.

Repeating the above elementary argument, we can get the following list

If $\beta$ is hyperbolic positive type one, then
\begin{equation*}
\Psi(\tau) = \begin{pmatrix}
a > 1 & b < 0\\
c < 0 & a > 1
\end{pmatrix}
\end{equation*}
or equivalently
\begin{equation*}
\Psi(\frac{\tau}{2}) = \begin{pmatrix}
u > 0 & v < 0\\
w < 0 & x > 0
\end{pmatrix}\ \text{or}\ \Psi(\frac{\tau}{2}) = \begin{pmatrix}
u < 0 & v > 0\\
w > 0 & x < 0
\end{pmatrix}
\end{equation*}

If $\beta$ is hyperbolic positive type two, then
\begin{equation*}
\Psi(\tau) = \begin{pmatrix}
a > 1 & b > 0\\
c > 0 & a > 1
\end{pmatrix}
\end{equation*}
or equivalently
\begin{equation*}
\Psi(\frac{\tau}{2}) = \begin{pmatrix}
u > 0 & v > 0\\
w > 0 & x > 0
\end{pmatrix}\ \text{or}\ \Psi(\frac{\tau}{2}) = \begin{pmatrix}
u < 0 & v < 0\\
w < 0 & x < 0
\end{pmatrix}
\end{equation*}

If $\beta$ is hyperbolic negative type one, then
\begin{equation*}
\Psi(\tau) = \begin{pmatrix}
a < -1 & b > 0\\
c > 0 & a < -1
\end{pmatrix}
\end{equation*}
or equivalently
\begin{equation*}
\Psi(\frac{\tau}{2}) = \begin{pmatrix}
u > 0 & v < 0\\
w > 0 & x < 0
\end{pmatrix}\ \text{or}\ \Psi(\frac{\tau}{2}) = \begin{pmatrix}
u < 0 & v > 0\\
w < 0 & x > 0
\end{pmatrix}
\end{equation*}

If $\beta$ is hyperbolic negative type two, then
\begin{equation*}
\Psi(\tau) = \begin{pmatrix}
a < -1 & b < 0\\
c < 0 & a < -1
\end{pmatrix}
\end{equation*}
or equivalently
\begin{equation*}
\Psi(\frac{\tau}{2}) = \begin{pmatrix}
u > 0 & v > 0\\
w < 0 & x < 0
\end{pmatrix}\ \text{or}\ \Psi(\frac{\tau}{2}) = \begin{pmatrix}
u < 0 & v < 0\\
w > 0 & x > 0
\end{pmatrix}
\end{equation*}

In the proof of Real ECH inequality we need to see the type of $\Psi_\beta(-\tau)$, because
\begin{equation}\label{Appendix}
\Psi(-\tau) = \Psi(\tau)^{-1} = \left(
     \begin{array}{cc}
       a & -b \\
       -c & a \\
     \end{array}
   \right)
\end{equation}
From the above list, we can see, $\Psi_\beta(\tau), \Psi_\beta(-\tau)$ have different type.

Any matrix in the conjugacy class $\{A\Psi(\tau) A^{-1}\}$, $A=\text{diag}(\varepsilon, \frac{1}{\varepsilon}), \varepsilon \in {\bf R}^+$(we need to fix the $x, y$-axis), has the same iteration formula pattern, which represent different trivializations. We can see it as the diagonal symplectic matrix group $D = \{A| A\ \text{is diagonal}, A \in Sp(2)\}$ acts on the manifold $C$ by conjugation, where $C =\{B \in Sp(2)\}$ consists of such symplectic matrix
\begin{equation*} B =
\begin{pmatrix}
a & b\\
c & a
\end{pmatrix}
\end{equation*}

We consider the GIT quotient of this group action which means two matrices $B_1, B_2$ are equivalent if and only if their orbit closures has nonempty intersection $\overline{\{CB_1C^{-1}\}} \bigcap \overline{\{CB_2C^{-1}\}} \neq \emptyset$. See \cite{Mumford} and \cite[section 10.5]{UrsOtto}.

First note that conjugation does not change the trace, $a$ is a invariant under conjugation.

Next we have the formula
\begin{equation*}
CBC^{-1} = \begin{pmatrix}
a & b\varepsilon^2\\
c\frac{1}{\varepsilon^2} & a
\end{pmatrix}
\end{equation*}
If $B$ is nondegenerate, $b,c \neq 0$, we always can choose $\varepsilon$ such that $b\varepsilon^2 = \pm c\frac{1}{\varepsilon^2}$.

If $B$ is elliptic, then $B$ is equivalent to
\begin{equation*}
R(\theta) =
\left(
 \begin{array}{cc}
 \cos\theta & -\sin\theta \\
 \sin\theta & \cos\theta \\
 \end{array}
 \right)
\end{equation*}

If $B$ is hyperbolic, then $B$ is equivalent to
\begin{equation*}
\left(
  \begin{array}{cc}
    a & \pm\sqrt{a^2 - 1} \\
    \pm\sqrt{a^2 - 1} & a \\
  \end{array}
\right), a > 1\ \text{or}\ a < -1
\end{equation*}

It is easy to see that
\begin{equation*}
  \left(
     \begin{array}{cc}
       1 & 0 \\
       0 & 1 \\
     \end{array}
   \right),
   \left(
     \begin{array}{cc}
       1 & \pm b \\
       0 & 1 \\
     \end{array}
   \right),
   \left(
     \begin{array}{cc}
       1 & 0 \\
       \pm b & 1 \\
     \end{array}
   \right), b> 0
\end{equation*}
are equivalent to each other in the GIT quotient.

Similarly, we have
\begin{equation*}
  \left(
     \begin{array}{cc}
       -1 & 0 \\
       0 & -1 \\
     \end{array}
   \right),
   \left(
     \begin{array}{cc}
       -1 & \pm b \\
       0 & -1 \\
     \end{array}
   \right),
   \left(
     \begin{array}{cc}
       -1 & 0 \\
       \pm b & -1 \\
     \end{array}
   \right), b > 0
\end{equation*}
are equivalent to each other in GIT quotient.

Topologically, the quotient is isomorphic to a circle with four spikes, each spike represents a hyperbolic subcase, and we can identify the quotient space as
\begin{equation*}
  \{z \in {\bf C}| |z| = 1 \} \bigcup \{z \in {\bf C}| z = x+ iy, x=-1 \} \bigcup \{z \in {\bf C}| z = x+ iy, x=1 \} \subset {\bf C}
\end{equation*}

An element $z= e^{i\theta} \in \{z \in {\bf C}| |z| = 1 \}$, $z$ is identified with $R(\theta)$

An element $z= x+ yi \in \{z \in {\bf C}| z = x+ iy, x=-1 \}$, $z$ is identified with
\begin{equation*}
  \left(
     \begin{array}{cc}
       -\sqrt{y^2 + 1} & y \\
       y & -\sqrt{y^2 + 1} \\
     \end{array}
   \right)
\end{equation*}
among them $x = -1, y > 0$ represents negative hyperbolic type one, $x = -1, y < 0$ negative hyperbolic type two.

An element $z= x+ yi \in \{z \in {\bf C}| z = x+ iy, x=1 \}$, $z$ is identified with
\begin{equation*}
  \left(
     \begin{array}{cc}
       \sqrt{y^2 + 1} & y \\
       y & \sqrt{y^2 + 1} \\
     \end{array}
   \right)
\end{equation*}
among them $x = 1, y < 0$ represents positive hyperbolic type one, $x = 1, y > 0$ positive hyperbolic type two.

\section*{Acknowledgements}
This paper is accomplished when the author is visiting University of Augsburg, Germany. The author is supported by China Scholarship Council(CSC) No.201806200130. Professor Frauenfelder has given many valuable advises to this paper. The author also thanks professor Frauenfelder's hospitality and fruitful discussions when he stays in Germany.

\bibliography{realech}

\end{document}